\documentclass[preprint,12pt]{elsarticle}

\usepackage[utf8]{inputenc}

\usepackage{tikz,environ,tkz-euclide, tikzscale}
\usetikzlibrary{math,arrows.meta,positioning,shapes,
fit,arrows,snakes,hobby}
\usepackage{enumitem}
\usepackage{float}
\usepackage{braket,amsfonts,amsthm}

\usepackage{array}

\usepackage[caption=false]{subfig}
\usepackage{pgfplots}
\usepackage{caption}

\newtheorem{theorem}{Theorem}[section]

\newtheorem{lemma}[theorem]{Lemma}
\newtheorem{corollary}[theorem]{Corollary}
\newtheorem{remark}[theorem]{Remark}
\newtheorem{proposition}[theorem]{Proposition}
\newtheorem{conjecture}[theorem]{Conjecture}
\newtheorem{problem}[theorem]{Problem}

\newcommand{\PSL}{{\mathrm{PSL}}}

\usepackage{algorithmic}

\usepackage{graphicx,epstopdf}

\usepackage{amsopn}

\usepackage{xspace}
\usepackage{bold-extra}
\usepackage[most]{tcolorbox}

\colorlet{texcscolor}{blue!50!black}
\colorlet{texemcolor}{red!70!black}
\colorlet{texpreamble}{red!70!black}
\colorlet{codebackground}{black!25!white!25}

\begin{document}
\begin{frontmatter}

 \title{Decycling cubic graphs\tnoteref{grants}}
  \tnotetext[grants]{The first and the third author were supported by the grant no.
 APVV-19-0308 of Slovak Research and Development Agency.  The first author
 was partially supported by VEGA 2/0078/20. The second author was supported by the grant GA\v{C}R 19-17314J
 of the Czech Science Foundation. The third author was partially supported by
 VEGA 1/0727/22.}

\author[1,2]{Roman Nedela}
 \ead{nedela@savbb.sk}
 \address[1]{Faculty of Applied Sciences, University of West Bohemia, Pilsen Czech republic}
 \address[2]{Mathematical Institute of Slovak Academy of Sciences, Bansk\'a Bystrica, Slovakia}

\author[3]{Michaela Seifrtová}
 \ead{mikina@kam.mff.cuni.cz}
 \address[3]{Faculty of Mathematics and Physics, Charles University, Prague, Czech republic}

 \author[4]{Martin \v{S}koviera\corref{cor1}}
 \ead{skoviera@dcs.fmph.uniba.sk}
 \address[4]{Faculty of Mathematics, Pysics, and Informatics, Comenius University, Bratislava, Slovakia}
 \cortext[cor1]{Corresponding author}

\begin{abstract}
A set of vertices of a graph $G$ is said to be decycling if its
removal leaves an acyclic subgraph. The size of a smallest
decycling set is the decycling number of $G$. Generally, at
least $\lceil(n+2)/4\rceil$ vertices have to be removed in
order to decycle a cubic graph on $n$ vertices. In 1979, Payan
and Sakarovitch proved that the decycling number of a
cyclically $4$-edge-connected cubic graph of order $n$ equals
$\lceil (n+2)/4\rceil$. In addition, they characterised the
structure of minimum decycling sets and their complements. If
$n\equiv 2\pmod4$, then $G$ has a decycling set which is
independent and its complement induces a tree. If $n\equiv
0\pmod4$, then one of two possibilities occurs: either $G$ has
an independent decycling set whose complement induces a forest
of two trees, or the decycling set is near-independent (which
means that it induces a single edge) and its complement induces
a tree. In this paper we strengthen the result of Payan and
Sakarovitch by proving that the latter possibility (a
near-independent set and a tree) can always be guaranteed.
Moreover, we relax the assumption of cyclic
$4$-edge-connectivity to a significantly weaker condition
expressed through the canonical decomposition of 3-connected
cubic graphs into cyclically $4$-edge-connected ones. Our
methods substantially use a surprising and seemingly distant
relationship between the decycling number and the maximum genus
of a cubic graph.
\end{abstract}

 %\subjclass[2020]{05C38, 05C10, 04C05}

 %\begin{MSCcodes}
 %05C38, 05C05, 05C69
 %\end{MSCcodes}

 \end{frontmatter}

% ------------------------------------------------------------------
% END HEADER
% ------------------------------------------------------------------

\section{Introduction}
Destroying all cycles of a graph by removing as few vertices as
possible is a natural and extensively studied problem in graph
theory and computer science, which can be traced back at least
as far as to Kirchhoff's work on spanning trees
\cite{kirchhoff}. The minimum number of vertices whose deletion
eliminates all cycles of a graph $G$ is its \textit{decycling
number}, denoted by $\phi(G)$, and the corresponding set is a
\textit{minimum decycling set}. (Here we follow the terminology
of \cite{Bau_Worm, Bei_Vandel}. In the literature, decycling
number is also known as \textit{feedback vertex number} and the
corresponding set of vertices is a \textit{feedback vertex
set}, see \cite{Speck} for instance.)

In contrast to a similar problem of eliminating cycles by
removing edges, which amounts to determining the cycle rank (or
the Betti number) $\beta(G)$ of a graph $G$, computing the
decycling number is long known to be a difficult problem. In
1972, it was proven by Karp \cite[Problem~7]{Karp72} that
finding a minimum decycling set, or in other words,
establishing the decycling number, is an NP-complete problem.
The problem remains NP-complete even when restricted to planar
graphs, bipartite graphs, or perfect graphs
\cite{Garey_Johnson, garey_johnson_stock}. On the other hand,
it becomes polynomial for a number of families, including
permutation graphs (Liang \cite{liang}), interval graphs,
cocomparability graphs,
%convex bipartite graphs (Liang and Chang \cite{liang_chang}),
and also for cubic graphs (Li and Liu \cite{liu_li}, Ueno et
al. \cite{U_K_Got}).

Evaluating the decycling number of a cubic graph is, in a
sense, well understood. It is not difficult to see that
$\phi(G)\ge\lceil(n+2)/4\rceil$ for every cubic graph $G$ of
order $n$. In 1988, Speckenmayer \cite{Speck} derived the
equation $\phi(G)=n/2-z(G)+1$, where the parameter $z(G)$
stands for the size of a maximum nonseparating independent set.
Surprisingly, the parameter $z(G)$ has an interpretation in
topological graph theory: it coincides with the maximum
orientable genus $\gamma_M(G)$ of $G$. The connection between
the decycling number and the maximum genus of a cubic graph was
revealed in 1997 by Huang and Liu \cite{HuangLiu} through the
formula $\phi(G)+\gamma_M(G)=n/2+1$ and was further developed
by Li and Liu \cite{liu_li}. Their results have a number of
important consequences. In particular, they imply that the
decycling number of a connected cubic graph reaches its natural
lower bound $\lceil(n+2)/4\rceil$ precisely when $G$ is
\textit{upper-embeddable}, that is, when $G$ has a $2$-cell
embedding into an orientable surface with at most two faces.
The following result follows from results of Huang and Liu
\cite{HuangLiu} and Long and Ren \cite{Long_Ren}. An
independent proof can be found in Section~3
(Theorem~\ref{thm:longren-2}). 

\begin{theorem}{\rm (Huang and Liu \cite{HuangLiu}, Long and Ren \cite{Long_Ren})}\label{thm:longren}
If $G$ is a connected cubic graph of order $n$, then
$\phi(G)\geq \lceil (n+2)/4\rceil$, and the equality holds if
and only if $G$ is upper-embeddable.
\end{theorem}

In 1975, Payan and Sakarovitch \cite{PS} proved that the
equality $\phi(G)=\lceil (n+2)/4\rceil$ holds for all
cyclically $4$-edge-connected graphs. In addition, they
determined the structure of the corresponding decycling sets.
They showed that every minimum decycling set $J$ of such a
graph is either independent or \textit{near-independent}
(meaning that it induces a subgraph with exactly one
edge), and its complement $A=V(G)-J$ induces a forest with at
most two components. Following the terminology established by
Payan and Sakarovitch we say that the partition $\{A,J\}$ of
$V(G)$ is a \textit{stable decycling partition} of $G$. If,
moreover, $A$ induces a tree, we say that $\{A,J\}$ is
\textit{coherent}, otherwise it is \textit{incoherent}.

Figure~\ref{examples} shows three simple examples of stable
decycling partitions. The vertices are coloured black or white,
with black vertices representing $A$ and white vertices
representing the minimum decycling set~$J$; the edges within
$A$ are solid, those within $J$ are dashed, and the edges
between the sets are dotted. Out of the two
partitions of the cube, one is coherent but the other is not.

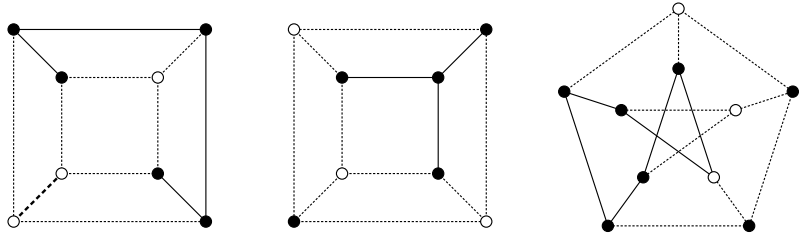
\begin{figure}
\centering\resizebox{0.8\textwidth}{!}{
\begin{tikzpicture}[every node/.style={inner sep=0,outer sep=0}, line cap=rect]  {

\tikzmath{\p=0.8; \q=0;\x=0.6;\w=0.6;\z=3.5;}
\coordinate (a11) at (\p,\q);
\coordinate (a12) at (\p+2*\w,\q);
\coordinate (a13) at (\p+2*\w,\q+2*\w);
\coordinate (a14) at (\p,\q+2*\w);
\coordinate (b11) at (\p-\x,\q-\x);
\coordinate (b12) at (\p+2*\w+\x,\q-\x);
\coordinate (b13) at (\p+2*\w+\x,\q+2*\w+\x);
\coordinate (b14) at (\p-\x,\q+2*\w+\x);

\draw[densely dotted] (a11)--(a12) ;
\draw[densely dotted] (a12)--(a13);
\draw[densely dotted] (a13)--(a14) ;
\draw[densely dotted] (a14)--(a11);

\draw [densely dotted] (b11)--(b12)node [midway, below, yshift=-12]{}; %{$Q_3=\{A,J\}$};
\draw (b12)--(b13);
\draw (b13)--(b14);
\draw [densely dotted] (b14)--(b11);

\draw[dotted, line width=0.3mm] (a11)--(b11);
\draw (a12)--(b12);
\draw[densely dotted] (a13)--(b13);
\draw (a14)--(b14);

\filldraw [fill=white] (a11) circle (2pt);
\filldraw  (a12) circle (2pt);
\filldraw [fill=white] (a13) circle (2pt);
\filldraw  (a14) circle (2pt);

\filldraw  [fill=white](b11) circle (2pt);
\filldraw  (b12) circle (2pt);
\filldraw  (b13) circle (2pt);
\filldraw  (b14) circle (2pt);

%-----------------------------------------------------------------------------------------

\coordinate (a11) at (\z+\p,\q);
\coordinate (a12) at (\z+\p+2*\w,\q);
\coordinate (a13) at (\z+\p+2*\w,\q+2*\w);
\coordinate (a14) at (\z+\p,\q+2*\w);
\coordinate (b11) at (\z+\p-\x,\q-\x);
\coordinate (b12) at (\z+\p+2*\w+\x,\q-\x);
\coordinate (b13) at (\z+\p+2*\w+\x,\q+2*\w+\x);
\coordinate (b14) at (\z+\p-\x,\q+2*\w+\x);

\draw[densely dotted] (a11)--(a12);
\draw (a12)--(a13);
\draw (a13)--(a14);
\draw [densely dotted](a14)--(a11);

\draw  [densely dotted](b11)--(b12)node [midway, below, yshift=-12]{}; %{$Q_3=\{A_1\cup A_2, J\}$};
\draw [densely dotted](b12)--(b13);
\draw [densely dotted](b13)--(b14);
\draw  [densely dotted](b14)--(b11);

\draw[densely dotted] (a11)--(b11);
\draw [densely dotted](a12)--(b12);
\draw (a13)--(b13);
\draw [densely dotted](a14)--(b14);

\filldraw [fill=white] (a11) circle (2pt);
\filldraw  (a12) circle (2pt);
\filldraw  (a13) circle (2pt);
\filldraw  (a14) circle (2pt);

\filldraw  (b11) circle (2pt);
\filldraw  [fill=white](b12) circle (2pt);
\filldraw  (b13) circle (2pt);
\filldraw  [fill=white](b14) circle (2pt);

%--------------------------------------------------------------------------------------------------------------------------------------------------------------------

\tikzmath{\p=0.8; \q=0.75;\x=0.6;\w=0.6;\z=8.5;}

\node (a10) {};
\path (a10) ++ (\z,\x/5*3+0.2) node (a11)  {};
\path (a11)  ++ (18:\q) node (b11){};
\path (a11)  ++(18+72:\q) node (b12){};
\path (a11)  ++(18+2*72:\q) node (b13){};
\path (a11)  ++(18+3*72:\q) node (b14){};
\path (a11)  ++(18+4*72:\q) node (b15){};
\path (a11)  ++(18:2*\q) node (c11){};
\path (a11)  ++(18+72:2*\q) node (c12){};
\path (a11)  ++(18+2*72:2*\q) node (c13){};
\path (a11)  ++(18+3*72:2*\q) node (c14){};
\path (a11)  ++(18+4*72:2*\q) node (c15){};

\draw [densely dotted](b11) -- (c11);
\draw [densely dotted](b12) -- (c12);
\draw (b13) -- (c13);
\draw (b14) -- (c14);
\draw [densely dotted](b15) -- (c15);

\draw [densely dotted](b11) -- (b13);
\draw (b13) -- (b15);
\draw (b15) -- (b12);
\draw (b12) -- (b14);
\draw [densely dotted](b14) -- (b11);

\draw [densely dotted](c11) -- (c12);
\draw [densely dotted](c12) -- (c13);
\draw (c13) -- (c14);
\draw [densely dotted](c14) -- (c15)node  [midway, below, yshift=-8]{}; 
%{\it{Petersen} $=\{A,J\}$};
\draw [densely dotted](c15) -- (c11);

\filldraw [fill=white] (b11) circle (2pt);
\filldraw  (b12) circle (2pt);
\filldraw  (b13) circle (2pt);
\filldraw  (b14) circle (2pt);
\filldraw  [fill=white](b15) circle (2pt);

\filldraw  (c11) circle (2pt);
\filldraw  [fill=white](c12) circle (2pt);
\filldraw  (c13) circle (2pt);
\filldraw  (c14) circle (2pt);
\filldraw  (c15) circle (2pt);

}
\end{tikzpicture}}
\caption{Examples of stable decycling partitions.}
    \label{examples}
\end{figure}

The theorem of Payan and Sakarovitch can be formulated as
follows.

\begin{theorem}{\rm (Payan and Sakarovitch \cite{PS})}\label{thm:PS}
Every cyclically $4$-edge-connected cubic graph has a stable
decycling partition. More precisely, if $G$ has $n$ vertices,
then the following hold:
\begin{enumerate}[label=\rm(\alph*)]
\item[{\rm(i)}] If $n\equiv2\pmod4$, then $G$ has a
    partition $\{A,J\}$ where $A$ induces a tree and $J$ is
    independent.
\item[{\rm(ii)}] If $n\equiv0\pmod4$, then $G$ has a
    partition $\{A,J\}$ where either
\item[] {\rm 1.} $A$ induces a tree and $J$ is
    near-independent, or
\item[] {\rm 2.} $A$ induces a forest of two trees and
    $J$ is independent.
\end{enumerate}
\end{theorem}

It is not difficult to realise that stable decycling partitions
exist even beyond  the class of cyclically $4$-edge-connected
cubic graphs. It is therefore natural to ask which cubic graphs
admit a stable decycling partition. We answer this question in
Theorems~\ref{thm:1-face} and~\ref{thm:2-face} stated in
Section~\ref{sec:partitions} which, put together, yield the
following statement.

\begin{theorem} \label{thm:stable}
A connected cubic graph admits a stable decycling partition if
and only if it is upper-embeddable.
\end{theorem}

In this context it is important to note that the concepts of
maximum orientable genus and upper-embeddability of graphs have
been extensively studied and are fairly well understood. There
exist results (for example in \cite{jung, KOK, nebesky2,
xuong1}) which provide min-max characterisations of maximum
genus in purely combinatorial terms, and there is a
polynomial-time algorithm to determine the maximum genus of an
arbitrary connected graph \cite{FGM, G}. As a consequence, the
Payan-Sakarovitch theorem readily follows from
Theorem~\ref{thm:stable} (more precisely from the more detailed
Theorems~\ref{thm:1-face} and~\ref{thm:2-face}) and the long
known fact that all cyclically $4$-edge-connected graphs are
upper-embeddable \cite{PX, KG, nebesky83}.

Stable decycling partitions provide a useful insight into the
structure of cubic graphs, which is the reason why they
naturally emerge in several different contexts beyond decycling
of cubic graphs or embedding cubic graphs into orientable
surfaces with high genus. For example, Glover and Marušič
\cite{Glover2007} used Theorem~\ref{thm:PS} to construct
Hamilton cycles or Hamilton paths in cubic Cayley graphs for a
rich family of quotients of the modular group
$\PSL(2,\mathbb{Z})$. Their technique was subsequently used in
a few other papers dealing with the hamiltonicity of cubic
Cayley graphs \cite{GKM09, GKMM12, KUTNAR20095491}. Stable
decycling partitions of cubic graphs play an essential role
also in a recent computer-assisted proof of the Barnette-Goodey
conjecture due to Kardo\v s \cite{kardos_computer}, which
states that simple planar cubic graphs with faces of size at
most 6 are hamiltonian.

Proofs of these results share the following idea of
topological nature. First, a cubic graph $G$ in question is
cellularly embedded into a closed surface. Next, a suitable
tree $T$ in the dual graph $G^*$ is identified. Under certain
conditions, the vertex set of $T$ extends to a stable decycling
partition of a cubic subgraph $H\subseteq G^*$. If this
partition is coherent, then, depending on the parity of the
Betti number of $H$, it is used to construct either a Hamilton
cycle or a Hamilton path of $G$, which traverse the boundary of
the union of faces of $G$ corresponding to the vertices of
$T\subseteq G^*$ (with possible exception of one edge).

Unfortunately, this idea does not work well if the decycling
partition is not coherent. From this point of view, coherent
decycling partitions are more valuable than incoherent ones.
Recall that for cyclically $4$-edge-connected cubic graphs
the theorem of Payan and Sakarovitch guarantees the
existence of a coherent decycling partition only when
$n\equiv2\pmod4$ (that is, when $\beta(G)$ is even). If
$n\equiv0\pmod4$, two possibilities can occur, only one of
which is a coherent decycling partition (namely the one in
Item~(ii)-1 of Theorem \ref{thm:PS}). In this situation one has
to ask under what conditions it is possible to ensure that a
cubic graph admits a coherent decycling partition if its Betti
number is odd (that is, if $n\equiv0\pmod4$). Nedela and
\v{S}koviera \cite{nedela} answered this question in the
positive provided that $G$ is cyclically $5$-edge-connected.
The main result of this paper replaces the assumption of cyclic
$5$-connectivity with cyclic $4$-connectivity, and thus
improves Theorem~\ref{thm:PS} by eliminating the possibility of
an incoherent decycling decomposition stated in Item~(ii)-2.

\begin{theorem}\label{thm:main}
Every cyclically $4$-edge-connected cubic graph $G$ admits a
coherent decycling partition. More precisely, the vertex set of
$G$ has a partition $\{A,J\}$ such that $A$ induces a tree and
$J$ is independent or near-independent, $J$ being independent
if and only if the order of $G$ equals $2\pmod4$.
\end{theorem}

What happens when the graph in question is not cyclically
$4$-edge-connec\-ted? First of all, there are many cubic graphs
with cyclic connectivity $3$ (including those with an odd Betti
number) that do admit a coherent decycling partition. In order
to get a deeper insight into the situation in $3$-connected
cubic graphs it is helpful to use the fact that every such
graph $G$ has a ``canonical'' decomposition $\{G_1, G_2,\ldots,
G_r\}$ into cyclically $4$-edge-connected cubic graphs. In
Section~\ref{sec:decomp} we prove that if at most one $G_i$ has
an odd Betti number, then $G$ admits a coherent decycling
partition (Theorem~\ref{thm:decomp1odd}). The family of
3-connected cubic graphs having at most one factor with an odd
Betti number includes the so-called odd-cyclically
$4$-connected graphs. Their characteristic property is that
every edge cut whose removal leaves a component with an odd
Betti number has size at least~$4$. It may be interesting to
mention that the concept of odd cyclic $4$-connectivity comes
from the study of maximum orientable genus where it was
independently introduced by Khomenko and Glukhov \cite{KG} and
Nebesk\'y \cite{nebesky83} in 1980 and 1983, respectively.

Nevertheless, there exist plenty of cubic graphs that do not
admit a coherent decycling partition. Theorem~\ref{thm:stable}
tells us that a necessary condition for a connected cubic graph
to admit a coherent decycling partition is to be
upper-embeddable. This condition is, however, not sufficient.
Examples of upper-embeddable cubic graphs of connectivity $1$
with no coherent decycling partition are not difficult to find:
the example in Figure \ref{c1_counter} can easily be turned
into an infinite family. Two-connected examples are less
apparent. Figure~\ref{fig:c2c_counter} displays two graphs
which are upper-embeddable, but none of them admits a coherent
decycling partition. An obvious generalisation of these two
graphs to an infinite family does not work because it produces
non-upper-embeddable graphs. In spite of that, an infinite
family does exist and is described in Section~\ref{sec:c2c}.

\begin{figure}[h]
\centering
\resizebox{0.4\textwidth}{!}{
\begin{tikzpicture}
[every node/.style={inner sep=0,outer sep=0}, line cap=rect]

  \tikzmath{\a=0.5; \b=1.45; \c=2.2; \d=2.7; \e=2.2; \f=1.9; \g=3.6;\h=3;\i=1.9;} %polomery
  \tikzmath{\u=90; \l=120; \m=60; \n=20; \o=20; \p=60;\q=5; \r=14;} %uhly
  \node (a1) {};
  
  \path (a1) ++(\u: \a) node (c11){};
  \path (a1) ++(\u+\l: \a) node (c12){};
  \path (a1) ++(\u+2*\l: \a) node (c13){};
  
  \path (a1) ++(\u: \b) node (c21){};
  \path (a1) ++(\u+\l: \b) node (c22){};
  \path (a1) ++(\u+2*\l: \b) node (c23){};

  %vonkajsie trojuholniky

  \path (a1) ++(\p+\u+\n: \f) node (e11){};
  \path (a1) ++(\p+\u+\l+\n: \f) node (e12){};
  \path (a1) ++(\p+\u+2*\l+\n: \f) node (e13){};
 
  \path (a1) ++(\p+\u-\n: \f) node (e21){};
  \path (a1) ++(\p+\u+\l-\n: \f) node (e22){};
  \path (a1) ++(\p+\u+2*\l-\n: \f) node (e23){};  

  \path (a1) ++(\u: \g) node (f11){};
  \path (a1) ++(\u+\l: \g) node (f12){};
  \path (a1) ++(\u+2*\l: \g) node (f13){};
  
  %------------vnutorne trojuhilnicky
  
  \path (a1) ++(\p+\u+\n+\o: \i) node (g11){};
  \path (a1) ++(\p+\u+\l+\n+\o: \i) node (g12){};
  \path (a1) ++(\p+\u+2*\l+\n+\o: \i) node (g13){};
 
  \path (a1) ++(\p+\u-\n-\o: \i) node (g21){};
  \path (a1) ++(\p+\u+\l-\n-\o: \i) node (g22){};
  \path (a1) ++(\p+\u+2*\l-\n-\o: \i) node (g23){};  
  
  \path (a1) ++(\u: \h) node (f31){};
  \path (a1) ++(\u+\l: \h) node (f32){};
  \path (a1) ++(\u+2*\l: \h) node (f33){};

  \draw [line width=0.4mm, dashed] (c11) -- (c12);
  \draw [line width=0.4mm] (c12) -- (c13);
  \draw [line width=0.4mm] (c11) -- (c13);  
  \draw [line width=0.4mm] (c11) -- (c21);
  \draw [line width=0.4mm] (c12) -- (c22);
  \draw [line width=0.4mm] (c13) -- (c23);  
   
  \draw [line width=0.4mm] (e13) -- (e21);  
  \draw [line width=0.4mm] (e12) -- (e23);  
  \draw [line width=0.4mm] (e11) -- (e22);  
  
  \draw [line width=0.4mm] (e11) -- (f12);  
  \draw [line width=0.4mm] (e12) -- (f13);  
  \draw [line width=0.4mm] (e13) -- (f11);  
  \draw [line width=0.4mm, dashed] (e21) -- (f11);  
  \draw [line width=0.4mm, dashed] (e22) -- (f12);  
  \draw [line width=0.4mm, dashed] (e23) -- (f13);  
  
  \draw [line width=0.4mm, dashed] (f31) -- (f11);  
  \draw [line width=0.4mm, dashed] (f32) -- (f12);  
  \draw [line width=0.4mm, dashed] (f33) -- (f13);  
  
  \draw [line width=0.4mm] (e11) -- (g11);  
  \draw [line width=0.4mm] (e12) -- (g12);  
  \draw [line width=0.4mm] (e13) -- (g13);      
  \draw [line width=0.4mm, dashed] (e21) -- (g21);  
  \draw [line width=0.4mm, dashed] (e22) -- (g22);  
  \draw [line width=0.4mm, dashed] (e23) -- (g23);  

  \draw [line width=0.4mm] (f31) -- (g13);  
  \draw [line width=0.4mm] (f32) -- (g11);
  \draw [line width=0.4mm] (f33) -- (g12);
  \draw [line width=0.4mm, dashed] (f31) -- (g21);  
  \draw [line width=0.4mm, dashed] (f32) -- (g22);
  \draw [line width=0.4mm, dashed] (f33) -- (g23);
  \draw [line width=0.4mm] (g13) -- (g21);  
  \draw [line width=0.4mm] (g11) -- (g22);
  \draw [line width=0.4mm] (g12) -- (g23);

  \filldraw (c11) circle (2pt);
  \filldraw (c12) circle (2pt);
  \filldraw (c13) circle (2pt);
  
  \filldraw (c21) circle (2pt);
  \filldraw (c22) circle (2pt);
  \filldraw (c23) circle (2pt);
  
  \filldraw (e11) circle (2pt);
  \filldraw (e12) circle (2pt);
  \filldraw (e13) circle (2pt);
 
  \filldraw (e21) circle (2pt);
  \filldraw (e22) circle (2pt);
  \filldraw (e23) circle (2pt);  
    
  \filldraw (f11) circle (2pt);
  \filldraw (f12) circle (2pt);
  \filldraw (f13) circle (2pt);

  \filldraw (g11) circle (2pt);
  \filldraw (g12) circle (2pt);
  \filldraw (g13) circle (2pt);
 
  \filldraw (g21) circle (2pt);
  \filldraw (g22) circle (2pt);
  \filldraw (g23) circle (2pt);  
  
  \filldraw (f31) circle (2pt);
  \filldraw (f32) circle (2pt);
  \filldraw (f33) circle (2pt);

 ----------

\end{tikzpicture}}
\caption{An upper-embeddable graph with no coherent partition.
Dashed lines represent the cotree edges of a Xuong tree
(see Section~\ref{sec:partitions} for the definition).}
    \label{c1_counter}
\end{figure}
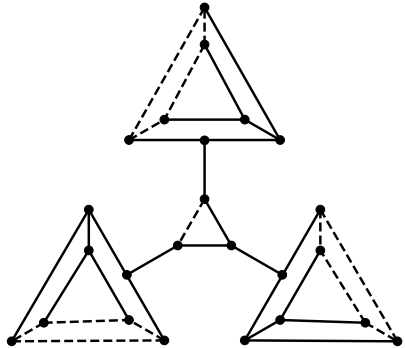

The situation in $3$-connected cubic graphs remains open: we
have no example of a $3$-connected upper-embeddable cubic graph
which would not have a coherent decycling partition. Motivated
by the difficulty of finding examples of $2$-connected
upper-embeddable cubic graphs with no coherent decycling
partition we propose the following conjecture.

\begin{figure}[h]
\centering
\resizebox{0.8\textwidth}{!}{
\input{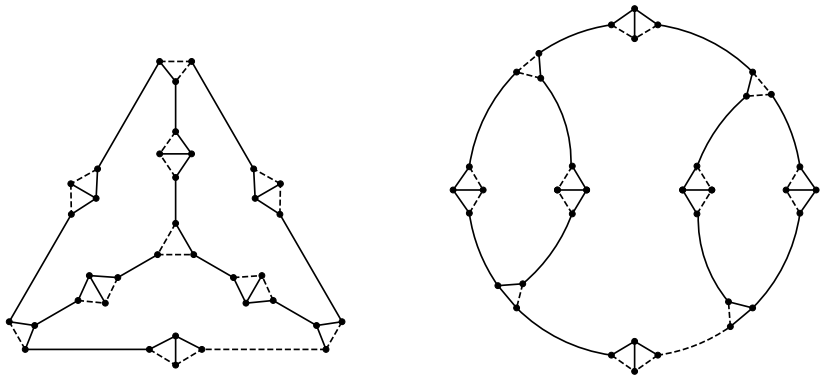}}
\caption{Two-connected upper-embeddable graphs $F_1$ (left) and $F_2$ (right) with no coherent
decycling partition. Dashed lines represent the cotree edges of a Xuong tree.
}
    \label{fig:c2c_counter}
\end{figure}

\begin{conjecture}\label{conj}
{\rm A $3$-connected cubic graph admits a coherent decycling
partition if and only if it is upper-embeddable. }
\end{conjecture}

There is yet another important aspect of decycling cubic
graphs. It concerns random graphs. As is well known \cite{RW},
a random cubic graph is hamiltonian (and therefore
$2$-connected). Bau et al. \cite{Bau_Worm} used this fact to
prove that a random cubic graph of order $n$ has decycling
number equal to $\lceil (n+2)/4\rceil$ almost surely. When
combined with Theorem~\ref{thm:longren}, this result states
that almost all cubic graphs are upper-embeddable. By contrast,
there is a positive probability that a random cubic graph
contains a triangle, which means that it cannot be cyclically
$4$-edge-connected (see Theorem 9.5 in \cite{Luczak}). Since all
upper-embeddable cubic graphs admit a stable decycling
partition, the following question suggests itself.

\begin{problem}\label{prob:as-amply}
{\rm Is it true that almost all cubic graphs admit a coherent
decycling partition? Equivalently, is it true that almost all
cubic graphs contain an induced tree such that the removal of
its vertices leaves a subgraph with at most one edge?}
\end{problem}

Our paper is organised as follows. In the next section we
collect the basic definitions required for understanding this
paper. In Section~\ref{sec:partitions} we reexamine the
relationship between minimum decycling partitions and maximum
orientable genus in cubic graphs. We prove
Theorem~\ref{thm:stable} about stable decycling partitions and
derive the result of Payan and Sakarovitch
(Theorem~\ref{thm:PS}) as a consequence. In
Section~\ref{sec:ample} we focus on coherent decycling
partitions. We introduce the concept of ample upper
embeddability and prove that a cubic graph with an odd Betti
number admits a coherent decycling partition if and only if it
is amply upper-embeddable (Theorem~\ref{thm:amply2-face}). In
Sections~\ref{sec:c4c} and~\ref{sec:decomp} we prove our main
results about coherent decycling partitions in cyclically
$4$-edge connected and certain cyclically $3$-edge-connected
cubic graphs. In Section~\ref{sec:c2c} we present an infinite
family of $2$-connected cubic graphs in which every stable
decycling partition is incoherent. The paper closes with a few
remarks and problems.

%----------------------------------------------------------------

\section{Preliminaries}\label{sec:cyc_connectivity}
\noindent{}Graphs in this paper will be mostly cubic
($3$-valent) or their subgraphs (subcubic). Loops and multiple
edges are permitted. Note that a $3$-connected cubic graph is
simple, which means that it has neither parallel edges nor
loops. We use the standard notation $V=V(G)$ for the vertex set
and $E=E(G)$ for the edge set of $G$. The symbol $n$ is
reserved for the \emph{order} of~$G$, the number of vertices of
$G$.

If a vertex $u$ of a cubic graph is incident with a loop, then
the other edge incident with $u$ is a bridge. In this paper,
both endvertices of the latter edge are regarded as
cutvertices, that is, including~$u$. This useful convention
comes from topological graph theory where graphs are
defined as finite $1$-dimensional CW-complexes and graph
connectivity corresponds to the connectivity of the underlying
topological space.

For a connected graph $G$, let $\beta(G) = |E|-|V|+1$ denote
its \emph{Betti number}, that is, the dimension of the cycle
space of $G$ (the \emph{cycle rank}), or in other words, the
number of edges of a cotree. Recall that a \emph{cotree}
of a spanning tree $T$ is the spanning subgraph $G-E(T)$ of
$G$. A graph whose Betti number is even is said to be
\emph{cyclically even}, otherwise it is \emph{cyclically odd}.
Observe that if $G$ is cubic, then $G$ is cyclically even
whenever $n\equiv2\pmod4$ and $G$ is cyclically odd whenever
$n\equiv0\pmod4$.

For a subset $A\subseteq V$, let $G[A]$ denote the subgraph
of $G$ induced  by $A$. If $A$ is a proper subset of vertices,
then $\delta_G(A)$ will denote the set of edges that join a
vertex in $A$ to a vertex in $V-A$. In other words,
$\delta_G(A)$ is the edge cut that separates $G[A]$
from $G[V-A]$. A $3$-edge-cut
$\delta_G(A)$ where $A$ or $V-A$ consists of a single vertex is
called \emph{trivial}. If $H=G[A]$, we usually write
$\delta_G(H)$ instead of $\delta_G(A)$.

We say that a set $B\subseteq E(G)$ is \emph{cycle-separating}
if $G-B$ is disconnected, and at least two of its components
contain cycles. Note that in a cubic graph $G$ a minimum cycle
separating subset of $E(G)$ is independent (that is a
matching), and that an independent edge cut is always cycle
separating. A connected graph $G$ is \emph{cyclically
$k$-edge-connected} if no set of fewer than $k$ edges is
cycle-separating in $G$. It can be shown that deleting any set
of $k\geq \beta(G)$ edges yields either a disconnected graph or
a graph without cycles. Thus, if $G$ contains a
cycle-separating set of edges, then it contains one with no
more than  $\beta(G)-1$ elements. We therefore define the
\emph{cyclic connectivity} of $G$, or more precisely,
\emph{cyclic edge connectivity}, denoted by $\zeta(G)$, as the
largest integer $k\leq\beta(G)$ for which $G$ is cyclically
$k$-edge-connected. For instance, $K_{3,3}$ is cyclically
$k$-edge-connected for every positive integer $k$, but its
cyclic connectivity equals~$4$. In fact, $\zeta(G)=\beta(G)$ if
and only if $G$ has no cycle-separating edge cut. If $G$ is
cubic, then it happens only for three graphs, namely for
$G\cong K_4$, $G\cong K_{3,3}$, and the dipole $D_2$ (the only
bridgeless cubic graph on two vertices).

Cyclic vertex connectivity can be defined in a similar manner;
for cubic graphs see \cite{MCCUAIG199216, atom}. It can be
shown that in cubic graphs, for $k\le 3$, $k$-connectivity,
$k$-edge-con\-nec\-tiv\-i\-ty, cyclic edge-connectivity, and
cyclic vertex-connectivity are all equal
\cite[Theorem~1.2]{MCCUAIG199216}. In other words, for cubic
graphs cyclic connectivity is a natural extension of the
classical connectivity invariants.

\section{Decycling partitions and maximum genus in cubic graphs}
\label{sec:partitions}
\noindent{}The purpose of this section
is to discuss a link between the existence of a stable
decycling partition of a cubic graph and its maximum orientable
genus. Our main aim is to prove that a connected cubic graph
has a stable decycling partition if and only if it is
upper-embeddable, that is, if it has a cellular embedding on an
orientable surface which has at most two faces. The importance
of this relationship is underscored by the fact that the
required partitions can be efficiently constructed because the
maximum orientable genus of every graph is polynomially
computable \cite{FGM, G}. This fact provides an alternative
proof of the result, proved in \cite{liu_li}, that the
decycling number of a cubic graph can be determined in
polynomial time.

We begin our exposition by reviewing basic facts about maximum
genus. For a deeper account we refer the reader to \cite{MoTh}
and \cite{White} and for a fairly recent survey to
\cite[Chapter~3]{topics}.

Throughout this paper, all embeddings will be cellular and
surfaces will be closed and orientable. Recall that a
\textit{cellular embedding} of a connected graph $G$ on a
closed surface $S$ is, roughly speaking, a drawing of $G$ on
$S$ without edge-crossings such that each component of $S-G$, a
\textit{face}, is homeomorphic to a disk (a $2$-cell). If a
graph with $n$ vertices and $m$ edges has a cellular embedding
on an orientable surface of genus $g$, then the number of
faces, denoted by $r$, must satisfy the Euler-Poincar\'e
equation
$$n-m+r=2-2g.$$

The \emph{maximum genus} of a connected graph $G$, denoted by
$\gamma_M(G)$, is the largest genus of an orientable surface
into which $G$ has a cellular embedding. It is an immediate
consequence of the Euler-Poincar\'e formula that
$\gamma_M(G)\le\lfloor\beta(G)/2\rfloor$, where $\beta(G)$ is
the Betti number of $G$. The quantity
$\xi(G)=\beta(G)-2\gamma_M(G)$ measures the distance of
$\gamma_M(G)$ from the natural upper bound of $\beta(G)/2$ and
is called the \emph{deficiency} of~$G$. It is easy to see that
$\xi(G)=r_\mathrm{min}(G)-1$ where $r_\mathrm{min}(G)$ is the
minimum number of faces among all cellular embedings of $G$.
The graphs for which $\gamma_M(G) = \lfloor\beta(G)/2\rfloor$
are called \emph{upper-embeddable}. Clearly, a graph is
upper-embeddable if and only if $\xi(G)\le 1$, or equivalently,
if it has a cellular embedding into an orientable surface with
one or two faces. In this context it may be worth mentioning
that every connected graph admits a cellular embedding with a
single face into some closed surface, possibly non-orientable
\cite{R, S}.

Upper-embeddable graphs thus fall into two types:
\emph{orientably one-face embeddable} graphs (whose Betti
number is even) and \emph{orientably two-face embeddable}
graphs (whose Betti number is odd); for brevity the adverb
``orientably'' will be omitted. In cubic graphs these two types
of upper-embeddability can easily be distinguished by the
number of vertices. A one-face embeddable cubic graph has
order $2\pmod4$ whereas a two-face embeddable cubic graph
has order $0\pmod4$.

One of the most remarkable facts about maximum genus is that
this topological invariant can be characterised in a purely
combinatorial manner. As early as in 1973, Khomenko et al.
\cite{KOK} proved that the maximum genus of an arbitrary
connected graph $G$ equals the maximum number of pairs of
adjacent edges whose removal leaves a connected spanning
subgraph of $G$. A similar characterisation was later
established by Xuong \cite{xuong1}. It states that $\xi(G)$
equals the minimum number of components with an odd number of
edges in a cotree of $G$. For brevity, such components will be
called \emph{odd}. Analogously, those with an even number of
edges will be called \emph{even}.

A spanning tree of $G$ whose cotree has exactly $\xi(G)$ odd
components is called a \emph{Xuong tree} of $G$. By a result of
Kotzig \cite[Theorem 4]{kotzig:1957}, every connected graph
with an even number of edges can be decomposed into pairs of
adjacent edges. Since a component with an odd number of edges
can be decomposed into pairs of adjacent edges and one
singleton, the cotree of a Xuong tree can be partitioned into
$\gamma_M(G)$ pairs of adjacent edges and $\xi(G)$ singletons.

For the sake of completeness we present a proof of Kotzig's
result.

\begin{lemma}\label{lm:kotzig}
{\rm (Kotzig \cite{kotzig:1957})}
The edge set of every connected graph with an even number of
edges can be partitioned into pairs of adjacent edges.
\end{lemma}

\begin{proof}
Let $G$ be a connected graph with an even number of edges. It
suffices to prove that $G$ can be oriented in such a way that
each vertex has an even in-degree. Indeed, for each vertex of
$G$ we simply distribute the incoming edges into disjoint
pairs, thereby producing a partition of the entire
edge set into pairs of adjacent edges.

Let us start with an arbitrary orientation $D$ of $G$. If $D$
has the required property, we are done. Otherwise, there are at
least two vertices with an odd in-degree, since $|E(G)|$ is
even. At one of them, say $v$, pick an edge directed outward
and continue to produce a directed walk $W$ starting at $v$
which extends as long as possible. Clearly, $W$ terminates at
another vertex with an odd in-degree, say $w$. After reversing
the orientation of all edges on $W$ the in-degree of $v$ and
$w$ becomes even while the other vertices of $G$ keep the
parity of their in-degree. We continue this process until all
vertices of $G$ obtain an even in-degree.
\end{proof}

We end this brief introduction to the theory of maximum genus
of a graph by stating two, in a sense complementary,
combinatorial characterisations of upper-embeddable graphs. The
first of them is due to Jungerman \cite{jung} and Xuong
\cite{xuong1}.

\begin{theorem}\label{thm:Xuong} {\rm (Jungerman \cite{jung}, Xuong \cite{xuong1})}
A connected graph is one-face embeddable if and only if it has
a spanning tree with no odd cotree components; it is two-face
embeddable if and only if it has a spanning tree whose cotree
has precisely one odd component.
\end{theorem}

The second characterisation of upper-embeddable graphs is due
to  Nebesk\'y \cite[Theorem 3]{nebesky2}. Let $G$ be a graph
and $X\subseteq E(G)$ an arbitrary set of edges. Let
$\text{ec}(G-X)$  and $\text{oc}(G-X)$ denote the number of
cyclically even and cyclically odd components of $G-X$,
respectively.

\begin{theorem}\label{thm:Nebesky} {\rm (Nebesk\'y \cite{nebesky2})}
A connected graph $G$ is upper-embeddable if and only if for an
arbitrary set $X\subseteq E(G)$ one has
\begin{equation}\label{eq:Nebesky}
\text{\rm ec}(G-X)+2\text{\rm oc}(G-X)-2\leq |X|.
\end{equation}
\end{theorem}

In the rest of this section we show that the two types of
upper-embeddability of cubic graphs lead to three different
types of decycling vertex partitions $\{A,J\}$, two of which
exist when the Betti number is odd, and one exists when the
Betti number is even. We will first treat the  case of even
Betti number, but before stating the result we present a lemma
which provides a convenient tool for dealing with cubic
upper-embeddable graphs in general.

Let $G$ be a connected cubic graph and $T$ a spanning tree of
$G$. Observe that each component of $G-E(T)$ is a path or a
cycle. A component of $G-E(T)$ will be called \emph{heavy} if
the number of its edges is at least three and has the
same parity as the Betti number of $G$. A component which is
not heavy will be called \emph{light}.

\begin{lemma}\label{lm:acyclic}
Let $G$ be a connected loopless cubic graph of order at
least $4$. If $G$ is upper-embeddable, then it has a Xuong tree
with acyclic cotree. If, in addition, $G$ has a Xuong tree with
a heavy cotree component, then it also has a Xuong tree with
acyclic cotree that contains a heavy component of the same
parity.
\end{lemma}

\begin{proof}
Take any Xuong tree $T$ of $G$. If $G-E(T)$ is acyclic, then
there is nothing to prove. If not, $G-E(T)$ has a component $Q$
which is a cycle.  Let $Q=(u_1u_2\ldots u_k)$, where $u_1, u_2,
\ldots, u_k$ are vertices of $G$ listed in cyclic ordering.
Since $G$ is loopless, we have $k\ge 2$. Each $u_i$ is adjacent
to some vertex $v_i$ such that the edge $u_iv_i$ belongs to
$T$. Note, however, that $v_i$ and $v_j$ need not be distinct
for $i\ne j$. Furthemore, the valency of each $v_i$ in $G-E(T)$
is at most $1$ for otherwise $u_iv_i$ would be the only edge of
$T$ and $G$ would have only two vertices. In particular, each
$v_i$ belongs to an acyclic component of $G-E(T)$. Set
$T'=T+u_1u_2-u_1v_1$; clearly, $T'$ is a spanning tree of $G$.
The cycle $Q$ in $G-E(T)$ now transforms into a
$u_2$-$v_1$-path $P$ in $G-E(T')$ of length $k$. If $Q$ is an
odd cycle, then $v_1$ must belong to an even component of
$G-E(T)$, because $T$ is a Xuong tree. Therefore $P$ becomes
part of an acyclic cotree component whose number of edges
is odd and at least $k$. If $Q$ is even, then $P$ becomes part
of an acyclic component of $G-E(T')$ which has the same parity
as the component of $G-E(T)$ containing $v_1$ and has at
least $k$ edges again. In other words, the transformation of
$T$ into $T'$ turns a heavy component of $G-E(T)$ into a heavy
acyclic component of $G-E(T')$. By repeating this process as
many times as necessary we arrive at a Xuong tree $T''$ of $G$
which has an acyclic cotree, and has a heavy cotree component
whenever the cotree of $T$ had.
\end{proof}

The following two theorems are the main results of this
section.

\begin{theorem}\label{thm:1-face}
For a connected cubic graph $G$ the following statements are
equivalent.
\begin{itemize}
\item[{\rm (i)}]$G$ has a cellular embedding on an
    orientable surface with one face.
\item[{\rm (ii)}] The vertex set of $G$ can be partitioned
    into two sets $A$ and $J$ such that $A$ induces a tree
    and $J$ is independent.
\end{itemize}
\end{theorem}

\begin{proof}
(i) $\Rightarrow$ (ii): By the Jungerman-Xuong Theorem, $G$ has
a spanning tree $T$ with all cotree components even.
Lemma~\ref{lm:kotzig} implies that $E(G)-E(T)$ has a partition
$\mathcal{P}$ into pairs of adjacent edges. For each pair
$\{e,f\}$ from $\mathcal{P}$ pick a vertex shared by $e$ and
$f$, and let $J$ be the set of all vertices obtained in this
way. Set $A=V(G)-J$. Since $G$ is cubic, each vertex of $J$
must be a pendant vertex of $T$. Observe that the vertices
chosen from distinct pairs of $\mathcal{P}$ must be
non-adjacent. So $J$ is an independent set of vertices and
$T-J$ is a tree with $V(T-J)=A$. Since every cotree edge with
respect to $T$ is incident with a vertex in $J$, the tree $T$
is an induced subgraph. In other words, $\{A,J\}$ is the
required vertex partition.

(ii) $\Rightarrow$ (i): Assume that the vertex set of $G$ has a
partition $\{A, J\}$ where $A$ induces a tree and $J$ is an
independent set. Let $S$ be the tree induced by $A$. For any
vertex $v\in J$ choose an arbitrary edge $e_v$ incident with
$v$ and form a spanning subgraph $S^+$ of $G$ by adding to $S$ all
the edges $e_v$ with $v\in J$. Since $J$ is independent, each
edge $e_v$ has its other end on $S$. Therefore $S^+$ is a
connected spanning subgraph of $G$. In fact, $S^+$ must be
acyclic, because each vertex $v\in J$ is joined to $S$ only by
the edge $e_v$. Thus $S^+$ is a spanning tree of $G$ in which
all vertices of $J$ are pendant. By the definition of $S^+$,
each cotree edge is incident with a single vertex of $J$, and
each vertex of $J$ is incident with precisely two cotree edges.
Hence the set of cotree edges can be partitioned into pairs of
adjacent edges and, consequently, each cotree component is even.
By Theorem~\ref{thm:Xuong}, $G$ is one-face embeddable.
\end{proof}

If a cubic graph $G$ has a vertex partition $\{A,J\}$ such that
$A$ induces a tree and $J$ is independent, then by
Theorem~\ref{thm:1-face} its Betti number must be even, and
hence its order is $2\pmod4$. The next theorem shows that if
the order of $G$ is $0\pmod4$, then there exists a similar
partition $\{A,J\}$ of the vertices of $G$, however, it is
either the independence of $J$ or the connectivity of the
subgraph induced by $A$ that fails.

\begin{theorem}\label{thm:2-face}
Let $G$ be a connected cubic graph. The following statements
are equivalent.
\begin{itemize}
\item[{\rm (i)}] $G$ has a cellular embedding on an
    orientable surface with two faces.

\item[{\rm (ii)}] The vertex set of $G$ can be partitioned
    into two sets $A$ and $J$ such that either

\item[] {\rm 1.} $A$ induces a tree and $J$ is
    near-independent, or
\item[] {\rm 2.} $A$ induces a forest with two components
    and $J$ is independent.
\end{itemize}
\end{theorem}

\begin{proof}
(i) $\Rightarrow$ (ii): First assume that $G$ contains a loop
incident with a vertex $v$, and let $u$ be the vertex adjacent
to $v$. Then $uv$ is a bridge and since deficiency is clearly
additive over bridges, $G-v$ is a one-face embeddable graph.
Let $G'$ be the cubic graph homeomorphic to $G-v$. By
Theorem~\ref{thm:1-face}, the vertex set of $G'$ has a vertex
partition $\{A,J\}$ where $A$ induces a tree and $J$ is
independent. Then $\{A\cup\{u\},J\cup\{v\}\}$ is a partition of
$V(G)$ such that $A$ induces a tree and $J$ is
near-independent.

Now let $G$ be loopless. Obviously, $G$ has at least four
vertices, so we can employ Lemma~\ref{lm:acyclic} to conclude
that $G$ has a Xuong tree $T$ with acyclic cotree. Let $B$ be
the unique odd component of $E(G)-E(T)$. Choose a pendant
edge $g=w_1w_2$ of $B$, where $w_2$ denotes a pendant vertex of
$B$. By applying Lemma~\ref{lm:kotzig} we decompose the set
$E(G)-(E(T)\cup\{g\})$ into pairs of adjacent edges. For each
pair $\{e,f\}$ from $\mathcal{P}$ we pick a vertex shared by
$e$ and $f$ and denote by $J'$ the resulting set of vertices.
Clearly, $J'$ is an independent set. For $i\in\{1,2\}$ set
$J_i=J'\cup\{w_i\}$ and $A_i=V(G)-J_i$. We show that both
$\{A_1,J_1\}$ and $\{A_2,J_2\}$ are vertex partitions
satisfying (ii) of the theorem.

First assume that $B$ is a heavy odd component of $E(G)-E(T)$,
and let $h=w_1w_3$ be the edge of $B$ incident with $w_1$ and
different from $g$. All the vertices of $J_1$ are now pendant
in $T$, so $T-J_1$ is a tree. Since every cotree edge is
incident with a vertex in $J_1$, the tree $T-J_1$ is induced
by~$A_1$. Observe that $\{w_1,w_3\}\subseteq J_1$ while $w_2\in
A_1$. As $J'$ is independent, we conclude that $h=w_1w_3$ is
the only edge joining a pair of vertices of $J_1$. Thus  $J_1$ is
near-independent and the partition $\{A_1,J_1\}$ satisfies
(ii)-1.

If $B$ has only one edge, namely $g=w_1w_2$, then the vertices
of $J_1$ except $w_1$ are pendant in $T$, and $w_1$ is a
$2$-valent vertex of $T$. Hence $T-J_1$ is a forest with two
components. Further, each cotree edge with respect to $T$ joins
a vertex in $J_1$ to a vertex in $A_1$, which implies that
$J_1$ is independent.

We remark that the partition $\{A_2,J_2\}$ is always of type
(ii)-2, no matter whether the unique odd cotree component is
heavy or not. In any case, the partitions $\{A_1,J_1\}$ and
$\{A_2,J_2\}$ fulfil (ii) of the theorem.

\medskip

(ii) $\Rightarrow$ (i): Assume that the vertex set of $G$ has a
partition $\{A,J\}$ satisfying the properties stated in (ii).
We distinguish two cases.

\medskip\noindent
Case 1. \textit{$A$ induces a tree and $J$ is a
near-independent set.} Let $h$ be the unique edge of the
subgraph $G[J]$, and let $S=G[A]$. Since $J$ is
near-independent, for each vertex $v$ of $J$ there exists at
least one edge joining $v$ to $S$. Choose one of them and
denote it by $t_v$. Form a spanning
subgraph $S^+$ of $G$ by adding to $S$ all the edges $t_v$
where $v\in J$. Clearly, $S^+$ is connected and acyclic, so
$S^+$ is a spanning tree of $G$ in which all the vertices of
$J$ are pendant.

If $h$ is not a loop, then $h=xy$ where $x$ and $y$ are
distinct vertices. It follows that each vertex $v$ of $J$ is
incident with two distinct cotree edges $e_v$ and $f_v$. If $v$
and $w$ are not adjacent in $G$, then the pairs
$\{e_v,f_v\}$ and $\{e_w,f_w\}$ are disjoint. However, for the
ends of $h$ we have $\{e_x,f_x\}\cap\{e_y,f_y\}=\{h\}$.
Assuming that $h=e_x=e_y$ we must conclude that $f_x\ne f_y$,
because otherwise $G[J]$ would contain both $h$ and $f_x=f_y$,
contrary to the assumption that $J$ is near-independent.
Therefore the pairs $\{e_v,f_v\}$, for $v\in J-\{x\}$, and the
singleton $\{f_x\}$ form a partition of $E(G)-E(S^+)$ which
shows that $S^+$ is a Xuong tree of $G$. In fact, the cotree
component containing $h$ is heavy.

If $h$ is a loop incident with a vertex $x$, then each vertex
$v$ of $J-\{x\}$ is incident with two distinct cotree edges
$e_v$ and $f_v$, and the pairs corresponding to distinct
vertices of $J-\{x\}$ are again disjoint. It follows that the
pairs $\{e_v,f_v\}$ for $v\in J-\{x\}$ and the singleton
$\{h\}$ form a partition of $E(G)-E(S^+)$ proving that $S^+$ is
a Xuong tree of $G$.

\medskip\noindent
Case 2. \textit{$A$ induces a forest with two components and
$J$ is independent.} Let $S_1$ and $S_2$ be the components of
the forest $G[A]$. Because $G$ is connected and $J$ is
independent, there must be a vertex $x\in J$ with a neighbour
$y_1$ in $S_1$ and a neighbour $y_2$ in $S_2$. For each vertex
$v\in J-\{x\}$ choose an arbitrary edge $t_v$ incident with
$v$, and form a spanning subgraph $T$ of $G$ extending $S_1\cup
S_2$ with the edges $xy_1$ and $xy_2$, and with all the edges
$t_v$ where $v\in J-\{x\}$. Clearly, $T$ is a spanning tree of
$G$. Observe that each vertex $v\in J-\{x\}$ is incident with
two cotree edges $e_v$ and $f_v$ while for $x$ there is a
single cotree edge $g$ incident with it. It follows that the
pairs $\{e_v,f_v\}$ for $v\in J-\{x\}$ together with the
singleton $\{g\}$ form a partition of $E(G)-E(T)$ proving that
$T$ is a Xuong tree of $G$.
\end{proof}

\begin{remark}
{\rm For loopless cubic graphs, the proof of
Theorem~\ref{thm:2-face} shows that the equivalence (i)
$\Leftrightarrow$ (ii) stated in the theorem holds true even
when Item (ii)-1 is removed from the statement. This fact,
combined with the theorem of Payan and Sakarovitch
(Theorem~\ref{thm:PS}), implies that for a cyclically
$4$-edge-connected cubic graph of order $n\equiv 0\pmod4$ one
can always guarantee the existence of a vertex partition
$\{A,J\}$ where $A$ induces a forest of two trees and $J$ is
independent. If we take into account Theorem~\ref{thm:1-face},
we can conclude that in a cyclically $4$-edge-connected cubic
graph $G$ there always exists a minimum decycling set $J$ which
is independent. In Section~\ref{sec:c4c} we show that $J$ can
also be chosen in such a way that $G-J$ is a tree (see
Theorem~\ref{thm:betterPS}). By Theorem~\ref{thm:1-face},
both properties can be achieved at the same time only when 
$n\equiv2\pmod{4}$.}
\end{remark}

As an immediate corollary of the previous two theorems we
obtain the following statement which is identical with
Theorem~\ref{thm:stable}.

\begin{corollary} \label{cor:stable}
A connected cubic graph admits a stable decycling partition if
and only if it is upper-embeddable.
\end{corollary}

The result of Payan and Sakarovitch \cite{PS} stated as
Theorem~\ref{thm:PS} now follows from the well known fact
\cite{PX, KG, nebesky83} that all cyclically $4$-edge-connected
graphs are upper-embeddable.

\begin{corollary}\label{cor:4cc}
Every cyclically $4$-edge-connected cubic graph admits a stable
decycling partition.
\end{corollary}

We finish this section by providing a different proof of the
result of Long and Ren \cite{Long_Ren} stated as
Theorem~\ref{thm:longren}.

\begin{theorem}\label{thm:longren-2}
If $G$ is a connected cubic graph of order $n$, then
$\phi(G)\geq \lceil (n+2)/4\rceil$, and the equality holds if
and only if $G$ is upper-embeddable.
\end{theorem}

\begin{proof}
Let $J$ be an arbitrary decycling set for $G$ and let
$A=V(G)-J$ be its complement. Denote by $e_A$ and $e_J$ the
numbers of edges of the induced subgraphs $G[A]$ and $G[J]$,
respectively. Since $G[A]$ is acyclic, we obtain
$$|\delta_G(A)|=3|A|-2e_A\geq 3(n-|J|)-2(n-|J|-1)=n-|J|+2.$$
On the other hand, $|\delta_G(A)|=|\delta_G(J)|=3|J|-2e_J$, so
\begin{equation}\label{eq:J}
3|J|\geq 3|J|-2e_J=|\delta_G(A)|\geq n-|J|+2,
\end{equation}
implying that $|J|\geq \lceil (n+2)/4\rceil$. Therefore $\phi(G)\geq \lceil(n+2)/4\rceil$.

We now prove that $\phi(G)=\lceil(n+2)/4\rceil$ if and only if $G$ is upper-embeddable. First assume that $\phi(G)=\lceil(n+2)/4\rceil$ and let $\{A,J\}$ be a decycling partition where the decycling set $J$ has $\lceil(n+2)/4\rceil$ vertices. If we insert $|J|=\lceil(n+2)/4\rceil$ into \eqref{eq:J},
we obtain $e_J=0$ if $n\equiv 2\pmod4$, and $e_J\leq 1$, if
$n\equiv 0\pmod4$. An easy counting argument now shows that
$G[A]$ has exactly one component either if  $n\equiv 2\pmod4$ or if
$e_J=1$ and $n\equiv 0\pmod4$. If $n\equiv 0\pmod4$ and $e_J=
0$, we derive that $G[A]$ has exactly two components. Hence, in
either case, the decycling partition $\{A,J\}$ is stable. By
Corollary~\ref{cor:stable}, $G$ is upper-embeddable.

Conversely, if $G$ is upper-embeddable, then, by
Corollary~\ref{cor:stable}, it admits a stable decycling
partition $\{A,J\}$. Direct counting yields that
$\phi(G)\le |J|=\lceil(n+2)/4\rceil$ for each of the three types of stable
partitions. As shown before, $\phi(G)\geq \lceil(n+2)/4\rceil$,  so 
$\phi(G)=\lceil(n+2)/4\rceil$, and the proof is complete.
\end{proof}

%--------------------------------------------------------------------------------------------------------

\section{Coherent decycling partitions and ample upper-embeddability}
\label{sec:ample}

\noindent{}Theorem~\ref{thm:1-face} tells us that a connected
cubic graph on $2\pmod4$ vertices has a coherent decycling
partition if and only if is one-face embeddable. By contrast,
two-face-embeddability is not sufficient for a cubic graph of
order $0\pmod4$ to have a coherent decycling partition, as
confirmed by the graphs in Figures~\ref{c1_counter}
and~\ref{fig:c2c_counter}. In this section we show that a cubic
graph of order $0\pmod4$ admits a coherent decycling partition
exactly when it possesses a stronger form of
upper-embeddability, one that allows removing a pair of
adjacent vertices without affecting upper-embeddability.

Let $G$ be an upper-embeddable cubic graph. A pair $\{x,y\}$ of
distinct vertices of $G$ will be called \emph{removable} if $x$
and $y$ are simply adjacent (that is, not connected by parallel
edges), the edge $xy$ is not a bridge, and $G-\{x,y\}$
remains upper-embeddable; otherwise a pair of adjacent vertices
is said to be \emph{non-removable}. It follows from the
definition that $G-\{x,y\}$ is connected. If $\{x,y\}$ is a
removable pair of vertices, then the graphs $G$ and $G-\{x,y\}$
have different Betti numbers. Therefore if a pair $\{x,y\}$
of vertices is removable, one of $G$ and $G-\{x,y\}$ is
one-face-embeddable and the other one is two-face-embeddable.

We say that a cubic graph $G$ is \emph{amply upper-embeddadble}
if it is upper-embeddable and contains a removable pair of
adjacent vertices. An upper-embeddable cubic graph is
\emph{tightly upper-embeddable} if it is not amply
upper-embeddable. In other words, in a tightly upper-embeddable
cubic graph the removal of every pair of simply adjacent
vertices produces either a disconnected graph or a graph with
deficiency at least $2$.

In this section we focus on amply upper-embeddable cubic graphs
with an odd Betti number, that is, with order $0\pmod 4$. For
obvious reasons we refer to them as \emph{amply two-face-embeddable} cubic graphs. A two-face-embeddable cubic graph
that is not amply two-face-embeddable is \emph{tightly
two-face-embeddable}.

Our main goal is to prove that a connected cubic graph of order
$0\pmod4$ admits a coherent decycling partition if and only if
it is amply two-face-embeddable, and to characterise such
graphs by providing a Jungerman-Xuong-type theorem for them.

We begin our study of amply upper-embeddable graphs with a
lemma showing that such graphs are loopless.

\begin{lemma}\label{lm:loop-tightly}
If an upper-embeddable cubic graph contains a loop, then it is
tightly two-face-embeddable.
\end{lemma}

\begin{proof}
Let $G$ be an upper-embeddable graph with a loop at a vertex
$u$ and let $v$ be the neighbour of $u$. Since the deficiency
is additive over bridges, the component $H$ of $G-uv$
containing $v$ is one-face embeddable, and hence the entire $G$
is two-face-embeddable. Suppose, to the contrary, that $G$
contains a removable pair of vertices, say $\{x,y\}$. Taking
into account that $uv$ is a bridge and $v$ is a cut-vertex of
$G$ we conclude that $\{x,y\}\cap\{u,v\}=\emptyset$. In
particular, both $x$ and $y$ are contained in $H$. We wish to
prove that $G-\{x,y\}$  has deficiency at least $2$. Since
$G-\{x,y\}$ is connected, so is $H-\{x,y\}$, given the position
of $x$ and $y$ within $G$. Recall that $\xi(H)=0$. It follows
that $H-\{x,y\}$ has an odd Betti number, and therefore
\mbox{$\xi(H-\{x,y\})\ge 1$}. Using the additivity of $\xi$ over
bridges again we eventually obtain that  \mbox{$\xi(G-\{x,y\})\ge 2$}, 
which contradicts the assumption that $\{x,y\}$ is removable.
Hence, $G$ is tightly upper-embeddable.
\end{proof}

The next lemma explores the relation of ample
upper-embeddability to the existence of heavy components in the
cotree of a Xuong tree.

\begin{lemma}\label{lm:amply}
Let $G$ be an upper-embeddable cubic graph. If $G$ contains a
Xuong tree whose cotree has a heavy component, then $G$ is
amply upper-embeddable.
\end{lemma}

\begin{proof}
Let $T$ be a Xuong tree of $G$ with a heavy cotree component
$K$. By Lemma~\ref{lm:acyclic} we may assume that $K$ is a
path. Since the length of $K$ is at least~$3$, $K$ contains an
edge $e=uv$ adjacent to a pendant edge of $K$. Consider the
graph $G'=G-\{u,v\}$. The vertices $u$ and $v$ are pendant
vertices of $T$, so $T'=T-\{u,v\}$ is a spanning tree of $G'$;
in particular, $G'$ is connected. If $\beta(G)$ is even, then
$K-\{u,v\}$ has a single non-trivial odd component. Hence $G'$
is upper-embeddable and $G$ is amply upper-embeddable. If
$\beta(G)$ is odd, then $K$ has an odd number of edges,
and $K-\{u,v\}$ consists of even components. It follows that
all components of $G'-E(T')$ are even, so $G'$ is
upper-embeddable, and therefore $G$ is again amply
upper-embeddable.
\end{proof}

Now we are ready for a characterisation of amply
upper-embeddable graphs with an odd Betti number.

\begin{theorem}\label{thm:amply2-face}
The following statements are equivalent for every connected
cubic graph~$G$.
\begin{itemize}
\item[{\rm (i)}] $G$ is amply two-face embeddable.

\item[{\rm (ii)}] $G$ has a Xuong tree with a single odd
    cotree component, which is heavy.

\item[{\rm (iii)}] $G$ admits a coherent
    decycling bipartition.
\end{itemize}
\end{theorem}

\begin{proof}
(i) $\Rightarrow$ (ii): Let $G$ be an amply two-face-embeddable
graph, and let $u$ and $v$ be a removable pair of vertices of
$G$. By Lemma~\ref{lm:loop-tightly}, none of $u$ and $v$ is
incident with a loop. Since $u$ and $v$ are not doubly
adjacent, $u$ has a neighbour different from $v$, and vice
versa. Let $u_1$ and $u_2$ be the other two neighbours of $u$
(possibly $u_1=u_2$) and let $v_1$ and $v_2$ be the other two
neighbours of $v$ (possibly $v_1=v_2$). Since $G$ has an odd
Betti number, the Betti number of $G'=G-\{u,v\}$ must be even,
and therefore $G'$ is one-face embeddable. By Xuong's Theorem,
$G'$ has a Xuong tree $T'$ with even cotree components. Extend
$T'$ by adding the edges $u_1u$ and $v_1v$ to $T'$. The
resulting subgraph $T$ is clearly a spanning tree of $G$. Note
that both $u_2$ and $v_2$ were contained in even components of
$G'-E(T')$. In $G-E(T)$ the vertices $u_2$ and $v_2$ lie in the
same component which includes the path $u_2uvv_2$. This
component is odd and heavy, while the remaining cotree
components remain even. Thus $T$ is the required spanning tree
of~$G$.

\medskip

(ii) $\Rightarrow$ (iii): Assume that $G$ has a Xuong tree $T$
with a single odd cotree component, which is heavy. Clearly,
$G$ must be loopless and of order at least~$4$. By
Lemma~\ref{lm:acyclic} we can assume that $E(G)-E(T)$ is
acyclic, so the only odd component $B$ of $G-E(T)$ is a path of
length at least three. We now construct a partition
$\mathcal{P}$ of $E(G)-E(T)$ into pairs of adjacent edges and
one singleton in such a way that the singleton is a pendant
edge of~$B$. For each pair $\{e,f\}$ from $\mathcal{P}$ pick a
vertex shared by both edges $e$ and $f$, while for the
singleton $\{g\}$ pick the end-vertex which is not a pendant
vertex of $B$. Let $J$ be the set of vertices obtained from
$\mathcal{P}$ in this way, and set $A'=V(G)-J$. Observe that
each vertex in $J$ is a pendant vertex of $T$, so the subgraph
induced by the set $A'$ is the tree $T-J$. Furthermore, the
subgraph induced by $J$ is near-independent and its single edge
is the edge of $B$ adjacent to $g$. This shows that $\{A',J\}$
is the the required partition of $G$.

\medskip

The implication (iii) $\Rightarrow$ (ii) was proved in Case~1
of the proof of Theorem~\ref{thm:2-face}, so it remains to
prove only the implication (ii) $\Rightarrow$ (i). Assume that
$G$ has a Xuong tree with a single odd cotree component, which
is heavy. Then $G$ is upper-embeddable, and by
Lemma~\ref{lm:amply} it is amply upper-embeddable.
\end{proof}

\begin{corollary}
Let $G$ be a connected cubic graph on $n$ vertices where
$n\equiv 0$ $\pmod 4$. If $G$ admits a Xuong tree with a unique
odd cotree component which has at least three edges, then $G$
admits a coherent stable decycling partition.
\end{corollary}

%--------------------------------------------------------------------------
\section{Coherent partitions of cyclically $4$-edge-connected graphs}
\label{sec:c4c}

\noindent{}In this section we prove that every cyclically
$4$-edge-connected cubic graph admits a coherent decycling
partition. This result strengthens Theorem~\ref{thm:PS} by
eliminating the second possibility in its item (ii). The proof
uses two important results, proved by Payan and Sakarovitch
\cite{PS} and by Wormald \cite{wormald}, respectively. The
former guarantees the existence of a stable decycling partition
$\{A,J\}$ where the decycling set $J$ contains an arbitrary
preassigned vertex. The latter result enables us to create
cyclically $4$-edge-connected cubic graphs from smaller ones by
adding an edge and increasing the order by $2$. The
corresponding operation is called an \emph{edge extension} and
is executed as follows. In a given cubic graph $G$ we take two
nonadjacent edges, subdivide each of them by a new vertex and
join the two resulting 2-valent vertices by a new edge. The
reverse operation is an \emph{edge reduction}.

\begin{theorem}{\rm (Payan and Sakarovitch \cite{PS})}\label{thm:Psvertex}
Let $G$ be a cyclically $4$-edge-connected cubic graph and let
$v$ be an arbitrary vertex of $G$. Then $G$ admits a stable
decycling partition $\{A,J\}$ such that $v$ belongs to the
decycling set $J$.
\end{theorem}

\begin{theorem}{\rm (Wormald \cite{wormald})}\label{wormald}
Every cyclically $4$-edge-connected cubic graph other than
$K_4$ and $Q_3$ can be obtained from a cyclically
$4$-edge-connected cubic graph with fewer vertices by an
edge-extension.
\end{theorem}

\begin{remark}
{\rm In \cite{PS}, Payan and Sakarovitch observed  that
Theorem~\ref{thm:Psvertex} does not hold for cubic graphs with
cyclic connectivity $3$.}
\end{remark}

We are now ready for the main result of this section.

\begin{theorem}\label{thm:betterPS}
Every cyclically $4$-edge-connected cubic graph $G$ has a
coherent stable decycling partition. More precisely, if $G$ has
$n$ vertices, then
\begin{enumerate}%[label=\rm(\alph*)]
\item[{\rm(i)}] for $n\equiv2\pmod4$ the vertex set of $G$
    has a partition $\{A,J\}$ where $A$ induces a tree and
    $J$ is independent;
\item[{\rm(ii)}] for $n\equiv0\pmod4$ the vertex set of $G$
    has a partition $\{A,J\}$ where $A$ induces a tree and
    $J$ is near-independent.
\end{enumerate}
\end{theorem}

\begin{proof}
From Section~\ref{sec:partitions} we already know that in a
cubic graph $G$ of order $n\equiv2\pmod4$ every stable
decycling partition is coherent. Thus we can assume that $G$ is
a cyclically $4$-edge-connected cubic graph of order
$n\equiv0\pmod4$. If $G$ is either $K_4$ and $Q_3$, then a
coherent decycling partition is easily found: one for $Q_3$ is
shown in Figure~\ref{examples}, and for $K_4$ any partition of
vertices into two $2$-element subsets is good. Now, let $G$ be
different from $K_4$ and $Q_3$. By Theorem~\ref{wormald},  $G$
arises from a cyclically $4$-edge-connected cubic graph $G'$ of
order $n-2$ by an edge extension. Two independent edges $xy$
and $wz$ of $G'$ are thus subdivided by vertices $u$ and $v$,
respectively, and $G$ is then created by adding the edge
$e=uv$. Since $G'$ is cyclically $4$-edge-connected and its
order is $2\pmod4$, it has a coherent decycling partition
$\{A',J'\}$ such that $A'$ induces a tree and $J'$ is
independent. Furthermore, according to
Theorem~\ref{thm:Psvertex}, we can assume that $x\in J'$. Let
$T'$ denote the tree induced by $A'$ in $G'$.

We show that $\{A'\cup\{v\}, J'\cup\{u\}\}$ is a coherent
decycling partition of $G$. Let us look at the distribution of
the vertices $y$, $w$, and $z$ with respect to the partition
$\{A',J'\}$. Since $x$ is in $J'$, $xy$ is an edge of $G'$, and
$J'$ is an independent set, we conclude that $y$ belongs to
$A'$. As regards the edge $wz$, either both its end-vertices
are contained in $A'$, or they belong to different sets of the
partition, say, $w\in J'$ and $z\in A'$. Observe that in either
case the induced subgraph $G[J'\cup\{u\}]$ has a single edge,
namely $xu$, so $J'\cup\{u\}$ is a near-independent set.
Furthermore, the subgraph $T=G[A'\cup\{v\}]$ is a tree: in the
former case it arises from $T'$ by subdividing its edge $wz$
with $v$, and in the latter case $T=T'+vz$.  In both cases
$\{A'\cup\{v\}, J'\cup\{u\}\}$ is a coherent decycling
partition of~$G$, as claimed.
\end{proof}

%---------------------------------------------------------------------------------------

\section{Coherent partitions of 3-connected cubic graphs}\label{sec:decomp}
\noindent{}In this section we strengthen
Theorem~\ref{thm:betterPS} by showing that coherent decycling
partitions can be guaranteed in rich classes of cubic graphs
whose cyclic connectivity equals $3$. The main idea behind is
the existence of a canonical decomposition of $3$-connected
cubic graphs into cyclically $4$-edge-connected factors.
Although such a decomposition is likely to belong to the area
of mathematical folklore, we need to work out the necessary
details before they can be applied to our main result.

Let $G_1$ and $G_2$ be bridgeless cubic graphs. Pick a vertex $v_1$ in
$G_1$ and a vertex $v_2$ in $G_2$, remove the two vertices,
retain the dangling edges (more precisely, free edge-ends), 
and identify each dangling edge of
$G_1-v_1$ with a dangling edge of $G_2-v_2$ to obtain a cubic
graph $G_1*G_2$. We say that $G_1*G_2$ is a \emph{$3$-sum} of
$G_1$ and $G_2$ with respect to $v_1$ and $v_2$, the two
vertices being the \emph{root vertices} for the $3$-sum.  The
three edges that connect the neigbours of $v_1$ to the
neighbours of $v_2$ form a $3$-edge-cut which we call the
\emph{principal edge cut} of $G_1*G_2$. Observe that if $G_1$
and $G_2$ are $3$-connected, so is $G_1*G_2$.

Conversely, let $G$ be a $3$-connected cubic graph containing a
$3$-edge-cut $S$ whose removal leaves two nontrivial components
$H_1$ and $H_2$; note that $S$ is necessarily cycle-separating,
and vice versa. We can turn each $H_i$ to a cubic graph $\bar
H_i$ by taking a new vertex $x_i$ and attaching the three
dangling edges of $H_i$ to it. Clearly, both $\bar H_1$ and
$\bar H_2$ are $3$-connected. Moreover, $G=\bar H_1 * \bar
H_2$. We  have thus decomposed $G$ into a $3$-sum of two
smaller $3$-connected graphs $\bar H_1$ and $\bar H_2$. If any
of $\bar H_1$ and $\bar H_2$ contains a cycle-separating
$3$-edge-cut, we can repeat the process and obtain a set of
three cubic graphs such that $G$ can be reconstructed from them
by using $3$-sum twice. After a finite number of steps we
produce a collection $\{G_1, G_2, \ldots, G_r\}$ of
$3$-connected cubic graphs such that $G$ can be reconstructed
from them by repeated application of $3$-sum, but none of the
graphs $G_i$ can be expressed as a $3$-sum of two smaller cubic
graphs. In other words, each $G_i$ is cyclically
$4$-edge-connected. The collection $\{G_1, G_2, \ldots, G_r\}$
is called a \emph{decomposition} of $G$ into a $3$-sum of
cyclically $4$-edge-connected graphs, and the graphs $G_i$ are
called the \emph{factors} of the decomposition. As we confirm
in following theorem, the collection $\{G_1, G_2,\ldots, G_r\}$
is determined uniquely up to isomorphism and permutation of
factors. In this sense, it is appropriate to call it a
\emph{canonical decomposition} of $G$.

\begin{theorem}\label{3-cut_decomp}
Every $3$-connected cubic graph $G$ has a decomposition
$\{G_1,\penalty0 G_2,\ldots,G_r\}$ into cyclically
$4$-edge-connected graphs such that $G$ can be reconstructed
from them by repeated $3$-sums. The decomposition is unique up
to isomorphism and order of factors.
\end{theorem}

\begin{proof}
The fact that every $3$-connected cubic graph can be decomposed
into cyclically $4$-edge-connected factors has been explained
prior to the formulation of the theorem. As regards the
uniqueness, it is not difficult to see that if $S_1$ and $S_2$
are two distinct cycle-separating $3$-edge-cuts, then the
result of the decomposition into three $3$-connected cubic
graphs by using the two cuts does not depend on the order in
which they are taken. Indeed, if $S_1\cap S_2=\emptyset$, the
conclusion is obvious. If $S_1\cap S_2\ne \emptyset$, then the
intersection consists of a single edge, and again both ways in
which the decomposition is performed produce the same set of
$3$-connected cubic graphs up to isomorphism.
\end{proof}

We are now ready for the main result of this section.

\begin{theorem}\label{thm:decomp1odd}
Every $3$-connected cubic graph whose canonical decomposition
into cyclically $4$-edge-connected cubic graphs contains at
most one cyclically odd factor admits a coherent decycling
partition.
\end{theorem}

\begin{proof}
Let $G$ be a $3$-connected cubic graph whose canonical
decomposition \linebreak $\{G_1,\dots,G_r\}$ into cyclically
$4$-edge-connected graphs contains at most one cyclically odd
factor, say $G_1$. We may assume that $r\ge 2$ for otherwise
the result directly follows from Theorem~\ref{thm:betterPS}.
Now, $G$ can be reconstructed from $\{G_1,\dots,G_r\}$ by a
repeated use of a $3$-sum. Furthermore, for cubic graphs $H$
and $K$ one has $\beta(H*K)=\beta(H)+\beta(K)-2$, so $G$ is
cyclically odd if and only if $G_1$ is cyclically odd.
Therefore, if we start the reconstruction of $G$ from $G_1$,
the statement of the theorem will follow immediately from
following the claim.

\medskip\noindent
Claim. \emph{Let $G=H*K$ be a $3$-sum of cubic graphs where $H$
admits a coherent decycling partition and $K$ is cyclically
$4$-edge-connected with $\beta(K)$ even. Then $G$ also admits a
coherent decycling partition.}

\medskip\noindent
Proof of Claim. Let $x$ be the root vertex of $H$, let $x_1$,
$x_2$, and $x_3$ be its neighbours, and let $y$ be the root
vertex of $K$ with neighbours $y_1$, $y_2$, and $y_3$. We
assume that the principal 3-edge-cut of $H*K$ consists of the
edges $x_1y_1$, $x_2y_2$, and $x_3y_3$.

Theorem~\ref{thm:1-face} and the fact that every cyclically
$4$-edge-connected cubic graph is upper-embeddable imply that
$K$ has a coherent decycling partition $\{A_K,J_K\}$ where
$J_K$ is independent. Let $\{A_H, J_H\}$ be an arbitrary
coherent decycling partition for $H$; recall that the induced
subgraph $H[J_H]$ has at most one edge. If such an edge exists,
we call it a \emph{surplus edge}. Let $T_H\subseteq H$ and
$T_K\subseteq K$ denote the trees induced by the sets $A_H$ and
$A_K$, respectively. Note that both trees are non-trivial,
because each of $H$ and $K$ has at least four vertices and
their decycling number is given by the formula $\lceil
n+2/4\rceil$, where $n$ is the number of vertices.

We distinguish two main cases depending on whether
$x\in J_H$ or $x\in A_H$.

\medskip\noindent
Case 1. \textit{The root vertex of $H$ belongs to $J_H$.}

\noindent We
choose the partition $\{A_K,J_K\}$ in such a way that $y\in
A_K$. This is indeed possible, because
Theorem~\ref{thm:Psvertex} tells us that we can pick a
neighbour $y'$ of $y$ and require that $y'\in J_K$; as a
consequence, $y$ lies $A_K$.  Now we can define a vertex
partition $\{A,J\}$ of $G$ by setting $A=A_H\cup (A_K-\{y\})$
and $J=(J_H-\{x\})\cup J_K$. To show that $\{A,J\}$ is a
coherent decycling partition we consider two subcases.

\medskip\noindent
Subcase 1.1. \textit{The root vertex of $H$ is incident with
the surplus edge.}

\noindent Without loss of generality we may assume that the
surplus edge of $H$ is $xx_3$. We now chose the partition
$\{A_K,J_K\}$ in such a way that $y_3\in J_K$. Since $T_K$ is
non-trivial, at least one of the edges incident with $y$
belongs to $T_K$, say $yy_1$, but possibly also $yy_2$. It
follows that $T_K-y$ is either a tree or it consists of two
trees, each containing a neighbour of $y$ different from~$y_3$.
It is easy to see that the induced subgraph $T=G[A]$ is a tree:
depending on the structure of $T_K-y$ either
$T=T_H\cup\{x_1y_1\}\cup (T_K-y)$ or
$T=T_H\cup\{x_1y_1,x_2y_2\}\cup(T_K-y)$. In the former case,
$J=(J_H-x)\cup J(K)$, while in the latter case $J=(J_H-x)\cup
(J(K)\cup\{y_2\})$. By the assumptions, if $e$ is an edge
joining two vertices of $J$, then $e\in\{x_1y_1, x_2y_2,
x_3y_3\}$. Clearly, $x_3y_3$ is an edge of $G[J]$. In the
former case, $x_1y_1\in E(T)$ and $x_2y_2$ joins a vertex of
$T$ to a vertex of $J$. In the latter case both $x_1y_1$ and
$x_2y_2$ belong to $T$. In either case,  $G[J]$ has precisely
one edge, namely $x_3y_3$. Hence, $\{A,J\}$ is a coherent
decycling partition of $G$.

\medskip\noindent
Subcase 1.2. \textit{The root vertex of $H$ is not incident
with the surplus edge.}

\noindent As before, we choose a coherent
decycling partition $\{A_K,J_K\}$ where $y\in A_K$, but now we
are not restricted to choose a suitable neighbour $y'$ of $y$
which should be in $J_K$. For simplicity we again choose
$y_3\in J_K$. We may also assume that either $xx_1$ only, or
both $xx_1$ and $xx_2$, belong to $T_K$. As in the previous
subcase, the induced subgraph $G[A]$ is a tree (with identical
description as above), and $G[J]$ has at most one edge. If such
an edge exists, then it coincides with the surplus edge of $H$.
We have thus proved that $\{A,J\}$ is a coherent decycling
partition of $G$.

\medskip\noindent
Case 2. \textit{The root vertex of $H$ belongs to $A_H$.}

\noindent In
this case we choose the coherent partition $\{A_K,J_K\}$ in
such a way that $y\in J_K$; the existence of such a partition
is guaranteed by Theorem~\ref{thm:betterPS}. Let $Y$ be the set
of all edges of $G$ having the form $x_iy_i$ where $x_ix$
belongs to $T_H$. Since $T_H$ is non-trivial, $1\le |Y|\le 3$.
Let us define a vertex partition $\{A,J\}$ of $G$ by setting
$A=(A_H-\{x\})\cup A_K$ and $J=J_H\cup(J_K-\{y\})$. Observe
that $T_H-x$ is a forest whose number of components equals
$|Y|$. It follows that $T=G[A]=(T_H-x)\cup Y\cup T_K$ is a
tree. Moreover, $G[J]$ has at most one edge, which, if it
exists, coincides with the surplus edge of $H$. Hence,
$\{A,J\}$ is a coherent decycling partition of $G$.

The proof is complete.
\end{proof}

In the remainder of this section we show that
Theorem~\ref{thm:decomp1odd} applies to the family of cubic
graphs defined by a  weaker form of cyclic $4$-connectivity
known as odd cyclic $4$-connectivity. A graph $G$ is said to be
\emph{odd-cyclically $k$-connected}, for an integer $k\ge 2$,
if every induced subgraph $H\subseteq G$ such that
$\delta_G(H)$ is cycle-separating and $\beta(H)$ is odd has
$|\delta_G(H)|\geq k$.  The concept of odd cyclic
$4$-connectivity was independently introduced by Khomenko and
Glukhov \cite{KG} and Nebesk\'y \cite{nebesky83} in 1980 and
1983, respectively, for the study of the maximum genus of a
graph.

Every cyclically $k$-edge-connected cubic graph is clearly
odd-cyclically $k$-connected, but the converse is false.
Figure~\ref{oddc4c_notc4c} depicts an example of a cubic graph
which is odd-cyclically $4$-connected while its standard cyclic
connectivity equals only $3$.

 \begin{figure}
    \centering
    \resizebox{0.3\textwidth}{!}{
    \begin{tikzpicture}
\tikzmath{\p=1; \q=0;\x=1; \w=1;\z=1.5;}
\coordinate (a1) at (\p,\q);
\coordinate (a2) at (\p,\q+\x);
\coordinate (b1) at (\p+\w,\q-0.4);
\coordinate (b2) at (\p+\w,\q-0.4+\x);
\coordinate (b3) at (\p+\w,\q-0.4+\x+\x);

\coordinate (c1) at (\p+\w+\z,\q-0.4);
\coordinate (c2) at (\p+\w+\z,\q-0.4+\x);
\coordinate (c3) at (\p+\w+\z,\q-0.4+\x+\x);
\coordinate (d1) at (\p+\w+\w+\z,\q);
\coordinate (d2) at (\p+\w+\w+\z,\q+\x);

\draw[-, line width=0.3mm] (a1)--(b1);
\draw[-, line width=0.3mm] (a1)--(b2);
\draw[-, line width=0.3mm] (a1)--(b3);
\draw[-, line width=0.3mm] (a2)--(b1);
\draw[-, line width=0.3mm] (a2)--(b2);
\draw[-, line width=0.3mm] (a2)--(b3);

\draw[-, line width=0.3mm] (c1)--(d1);
\draw[-, line width=0.3mm] (c1)--(d2);
\draw[-, line width=0.3mm] (c2)--(d1);
\draw[-, line width=0.3mm] (c2)--(d2);
\draw[-, line width=0.3mm] (c3)--(d1);
\draw[-, line width=0.3mm] (c3)--(d2);

\draw[-, line width=0.3mm] (b1)--(c1);
\draw[-, line width=0.3mm] (b2)--(c2);
\draw[-, line width=0.3mm] (b3)--(c3);

\filldraw  (a1) circle (2pt);
\filldraw  (a2) circle (2pt);
\filldraw  (b1) circle (2pt);
\filldraw  (b2) circle (2pt);
\filldraw  (b3) circle (2pt);

\filldraw  (c1) circle (2pt);
\filldraw  (c2) circle (2pt);
\filldraw  (c3) circle (2pt);
\filldraw  (d1) circle (2pt);
\filldraw  (d2) circle (2pt);

\end{tikzpicture}}
    \caption{An odd-cyclically $4$-connected cubic graph which
    is not cyclically $4$-edge-connected.}
    \label{oddc4c_notc4c}
    \end{figure}
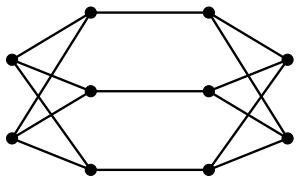

\begin{theorem}\label{thm:PS_odd_c4c}
Let $\{G_1, G_2,\ldots, G_r\}$ be the canonical decomposition
of an odd-cyclically $4$-connected cubic graph with cyclic
connectivity $3$. Then each $G_i$ has an even Betti number, and
so has $G$. In particular, $G$ has a coherent decycling
partition $\{A,J\}$ where $A$ induces a tree and $J$ is an
independent set.
\end{theorem}

\begin{proof}
It is sufficient to prove that if $H$ is a $3$-connected
odd-cyclically $4$-connected cubic graph which can be expressed
as a $3$-sum $K*L$ of cubic graphs $K$ and $L$, then both $K$
and $L$ are odd-cyclically $4$-connected and have an even Betti
number. By symmetry, it suffices to prove it for, say,~$K$.

Since $H=K*L$, there exists a vertex $u$ in $K$ such that $K-u$
is an induced subgraph of $H$ with $|\delta_H(K-u)|=3$. By the
definition of odd-cyclic $4$-connectivity, $\beta(K-u)$ is
even. It follows that $\beta(K)$ is even, similarly $\beta(L)$
is even, and consequently $\beta(H)=\beta(K)+\beta(L)-2$ is
also even.

Next we prove that $K$ is odd-cyclically $4$-connected. Suppose
not. Then $K$ contains an induced subgraph $Q$ such that
$\beta(Q)$ is odd and $|\delta_K(Q)|\le 3$. Since $K$ is
$3$-connected, we infer that $|\delta_K(Q)|=3$. Let $Q'$ be the
other component of the graph $K-\delta_K(Q)$, that is,
$Q'=K-V(Q)$. Further, let $u$ be the root vertex of $K$ used
for the $3$-sum producing~$H$. We can clearly assume that $u$
belongs to $Q'$. Note that $\beta(Q')$ is odd, too. Now, at
most one edge of $\delta_K(Q)$ is incident with $u$, because
$\delta_K(Q)$ is a minimum cycle-separating edge cut. However,
irrespectively of whether $\delta_K(Q)$ does or does not
contain such an edge, it is easy to see that $Q$ is an induced
subgraph of $H$ with $\beta(Q)$ odd and $|\delta_H(Q)|=3$,
contradicting the assumption that $H$ is odd-cyclically
$4$-connected. This completes the proof.
\end{proof}

%---------------------------------------------------------------------------------

\section{Cubic graphs with no coherent decycling partitions}\label{sec:c2c}

\noindent{}The aim of this section is to present an infinite
family of 2-connected cubic graphs which have a stable
decycling partition but no coherent one. By
Theorem~\ref{thm:amply2-face}, this amounts to displaying an
infinite family $\mathcal{F}$ of $2$-connected tightly
upper-embeddable graphs. For convenience, the topological
language will be in the foreground throughout.

The family $\mathcal{F}$ will be built from two graphs
displayed in Figure~\ref{fig:c2c_counter}. An important fact
about them is that they are \emph{claw-free}, which means that
they do not contain an induced subgraph isomorphic to the
complete bipartite graph $K_{1,3}$. In fact, the entire family
$\mathcal{F}$ consists of claw-free graphs.

Simple $2$-connected claw-free cubic graphs were characterised
by Palmer et al. \cite{Palmer} (see \cite[Proposition~1]{Oum} in 2002). Their
characterisation requires two operations. The first of them is
the well-known \emph{inflation} of a vertex to a triangle. It
replaces a vertex $v$ with a triangle $W_v$ and attaches the
edges formerly incident with $v$ to distinct vertices of $W_v$.
The second operation is a \emph{diamond insertion}, where a
\emph{diamond} means a graph isomorphic to $K_4$ minus an edge.
This operation is performed by replacing an edge $e=uv$ with a
path $P_e=uu'v'v$ of length $3$ and substituting the inner edge
$u'v'$ of $P_e$  with a diamond $D_{e}$ in such a way that its
$2$-valent vertices are identified with $u'$ and $v'$,
respectively. More specifically, this is an \emph{insertion of
a diamond} into $e$, see Figure~\ref{fig:insertion}

\begin{figure}[h]
\centering
\resizebox{0.7\textwidth}{!}{
\begin{tikzpicture}
[every node/.style={inner sep=0,outer sep=0}, line cap=rect, scale=0.7]  

  \node(a0) {};
  \path (a0)++ (0,0) node (a1){};
  \path (a0)++ (1,0) node (a2) {};
  \path (a0)++ (-1,0) node (a5) {};
  \path (a0)++ (-2,0) node (a6) {};  
  \path (a0)++ (90: 0.7) node(a3){};
  \path (a0)++ (270: 0.7) node(a4){};
  
  \path (a6)++ (90+45:1) node (b6){};
  \path (a6)++ (180+45:1) node (b7){};  

  \path (a1)++ (90-45:1) node (e6){};
  \path (a1)++ (-45:1) node (e7){};

  \filldraw (a6) circle (3pt)node[above, yshift=5] {$u$};
  \filldraw (a1) circle (3pt)node[above, yshift=5] {$v$};

  \draw [-, line width=0.4mm](a1) -- (a6)node[midway,above,yshift=5]{$e$}; 
  \draw [-, line width=0.4mm](a1) -- (e6);   
  \draw [-, line width=0.4mm](a1) -- (e7);   
  \draw [-, line width=0.4mm](a6) -- (b6);   
  \draw [-, line width=0.4mm](a6) -- (b7);      
%------------------------------------------------------------------------  

  \path(a0) ++(3.8,0) node (b0) {};
  \path(b0)++ (-1.74,0) node (a8){};
  \path(b0)++ (0.75,0) node (a9){};  
    \draw[-{Latex}] (a8) -- (a9)node[midway, above, yshift=2]{};

%------------------------------------------------------------------------  

  \path(a0) ++(8.6,0) node (c0) {};
  \path (c0)++ (2.5,0) node (a1){};
  \path (c0)++ (1,0) node (a2) {};
  \path (c0)++ (-1,0) node (a5) {};
  \path (c0)++ (-2.5,0) node (a6) {};  
  \path (c0)++ (90: 0.7) node(a3){};
  \path (c0)++ (270: 0.7) node(a4){};
  
  \path (a6)++ (90+45:1) node (b6){};
  \path (a6)++ (180+45:1) node (b7){};  

  \path (a1)++ (90-45:1) node (e6){};
  \path (a1)++ (-45:1) node (e7){};

  \filldraw (a2) circle (3pt)node[above, yshift=5] {$v'$};
  \filldraw (a3) circle (3pt)node[above, yshift=5] {$s$};
  \filldraw (a4) circle (3pt)node[below, yshift=-5] {$t$};
  \filldraw (a5) circle (3pt)node[above, yshift=5] {$u'$};
  \filldraw (a6) circle (3pt)node[above, yshift=5] {$u$};
  \filldraw (a1) circle (3pt)node[above, yshift=5] {$v$};

  \draw [-, line width=0.4mm](a1) -- (a2);  
  \draw [-, line width=0.4mm](a2) -- (a3);  
  \draw [-, line width=0.4mm](a2) -- (a4);   
  \draw [-, line width=0.4mm](a3) -- (a4);  
  \draw [-, line width=0.4mm](a4) -- (a5);  
  \draw [-, line width=0.4mm](a5) -- (a3); 
  \draw [-, line width=0.4mm](a5) -- (a6);   
  \draw [-, line width=0.4mm](a1) -- (e6);   
  \draw [-, line width=0.4mm](a1) -- (e7);   
  \draw [-, line width=0.4mm](a6) -- (b6);   
  \draw [-, line width=0.4mm](a6) -- (b7);      
  
%------------------------------------------------------------------------  
\end{tikzpicture}}
\caption{Insertion of a diamond into an edge.}
    \label{fig:insertion}
\end{figure}
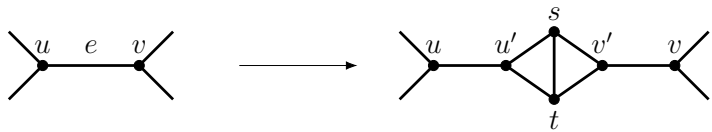

We can repeat the operation of inserting a diamond into the
edge $e=uv$ by inserting a diamond into the edge $v'v$, and so
on, which results in replacing $e$ with a \emph{string of
diamonds}. To be more explicit, we say that $e$ is replaced
with a \emph{string of $k$ diamonds}, where $k\ge 0$, if $e$ is
replaced with an alternating sequence $uu_1, D_1, v_1u_2, D_2,
\ldots, v_{k-1}u_k, D_k, v_kv$ of edges and diamonds such that
the $2$-valent vertices of the diamond $D_i$ are $u_i$ and
$v_i$ for each $i\in\{1,2,\ldots,k\}$. The number $k$ is the
\emph{length} of the string. The string collapses to the edge
$e=uv$ if $k=0$.

To finish the series of definitions needed to characterise
simple $2$-connected claw-free cubic graphs we define a
\emph{ring of diamonds} to be a graph obtained from an even
cycle by substituting every second edge with a diamond.

The characterisation of simple $2$-connected claw-free cubic graphs
reads as follows.

\begin{proposition}\label{prop:claw-free}
{\rm (Palmer \cite{Palmer})} Let $G$ be a simple $2$-connected
cubic graph. Then $G$ is claw-free cubic graph if and only if
one of the following holds:
\begin{itemize}
\item[{\rm (i)}] $G$ is isomorphic to $K_4$,

\item[{\rm (ii)}] $G$ is a ring of diamonds,

\item[{\rm (iii)}] $G$ is obtained from a
    $2$-edge-connected cubic graph $H$ by inflating each
    vertex of $H$ to a triangle and by replacing certain
    edges of $H$ with strings of diamonds.
\end{itemize}
\end{proposition}

We continue by defining the family $\mathcal{F}$.

\medskip\noindent
\textbf{Construction.} Take a 2-connected cubic graph $K$ on
four vertices; there are two such graphs -- the complete graph
$K_4$ on four vertices and the \textit{necklace} $L_4$ obtained
from a $4$-cycle by replacing every second edge with a pair of
parallel edges. In $K$, inflate each vertex to a triangle;
these triangles are called \emph{vertex-triangles}. In the
resulting graph $K'$ replace each edge not lying on a
vertex-triangle with a  string of diamonds with positive
length; different edges of $K'$ may accommodate different
numbers of diamonds. Include each graph obtained in this way in
$\mathcal{F}$. The smallest graphs in $\mathcal{F}$ are the
graphs $F_1$ and $F_2$ from Figure~\ref{fig:c2c_counter}, which
are obtained by vertex inflation and diamond insertion from
$K_4$ or to $L_4$, respectively.

\medskip

Our aim is to prove that all graphs in $\mathcal{F}$ are
tightly two-face embeddable. In the course of the proof it will
be convenient to say that a Xuong tree of a two-face-embeddable
cubic graph is \emph{light} if the unique odd component of the
corresponding cotree is light, that is, it is formed by a
single edge. It follows from Theorem~\ref{thm:amply2-face} that
every Xuong tree of a tightly two-face-embeddable cubic graph
is light. A Xuong tree which is not light will be called
\emph{heavy}.

We need the following lemma.

\begin{lemma}\label{lm:light}
If a cubic graph $G$ contains a light Xuong tree and $G'$
arises from $G$ by a diamond insertion, then $G'$ also contains
a light Xuong tree.
\end{lemma}

\begin{proof}
Assume that $G'$ arises from $G$ by inserting a diamond $D$
into an edge $e=uv$ of $G$. Let $u'$ be the $2$-valent vertex
of $D$ adjacent to $u$ in $G'$ and similarly let $v'$ be the
$2$-valent vertex of $D$ adjacent to $v$. Let $s$ and $t$
denote the two $3$-valent vertices of $D$.

Take an an arbitrary light Xuong tree $T$ of $G$. We show that
$T$ can be modified to a light Xuong tree of $G'$. We
distinguish three cases depending on the position of $e$ with
respect to $T$.

\medskip\noindent
Case 1. \textit{The edge $e$ lies in $T$.} In this case we
extend $T-e$ with the edges  $uu'$, $u's$, $st$, $sv'$, and
$v'v$ to obtain a spanning tree $T'$ of $G'$. It is easy to see
that $T'$ is light.

\medskip\noindent
Case 2. \textit{The edge $e$ lies in an even cotree component
with respect to $T$.} Recall that every even cotree component
with respect to $T$ can be partitioned into pairs of adjacent
edges. Let $f$ be the edge forming a pair with $e$, and let $Q$
be the component of $G-E(T)$ containing $e$ and $f$. Without loss
of generality we may assume that $v$ is the common endvertex of
both $e$ and $f$.  Now we extend $T$ with the edges $uu'$,
$u's$, $st$, and $sv'$ to obtain a spanning tree $T'$ of
$G'$. Observe that $Q$ is now transformed to an even component
of $G''$ containing the path $u'tv'v$ and the edge $f$, while
the odd component of $G-E(T)$ remains in $G'-E(T')$ intact.
Hence, $T'$ is a light Xuong tree of~$G'$.

\medskip\noindent
Case 3. \textit{The edge $e$ coincides with the one that forms
the odd component of $G-E(T)$.}  Now we form $T'$ by adding to
$T$ the path $uu'sv't$ of $G'$. As a consequence, the path
$u'ts$ forms an even component of $G'-E(T')$ and the edge $v'v$
constitutes the only odd component of $G'-E(T')$. Again, $T'$
is a light Xuong tree of $G'$.
\end{proof}

We proceed to the main result of this section. It characterises
all tightly two-face embeddable graphs within the class of
simple $2$-connected claw-free cubic graphs.

\begin{theorem}\label{thm:tight-family}
The following statements are equivalent for every simple
$2$-connected claw-free cubic graph $G$.
\begin{itemize}
\item[{\rm (i)}] $G$ is tightly two-face-embeddable.

\item[{\rm (ii)}] $G\in\mathcal{F}$.
\end{itemize}
\end{theorem}

\begin{proof}
(i) $\Rightarrow$ (ii): Assume that $G$ is a simple
$2$-connected claw-free cubic graph which is tightly
two-face-embeddable. Then one of the cases (i), (ii) or (iii)
of  Proposition~\ref{prop:claw-free} occurs. In cases (i) and
(ii) it is easy to find a heavy Xuong tree, so $G$ belongs to
the family of graphs characterised by (iii) of
Proposition~\ref{prop:claw-free}. It follows that $G$ arises
from a $2$-connected cubic graph $H$ by inflating every vertex
of $H$ to a triangle and by replacing certain edges of $H$ with
strings of diamonds.

At first we show that $H$ has four vertices. Let $n$ denote the
order of $H$. We apply Nebesk\'y's characterisation of
upper-embeddable graphs (Theorem~\ref{thm:Nebesky}) to $G$. Let
$X$ be the set of all edges of $G$ not lying on a triangle.
With this choice, each cyclically odd component of $G-X$ is a
vertex-triangle and each cyclically even component is a
diamond. Clearly, $\text{oc}(G-X)=n$. To calculate the number
of cyclically even components we count the contribution of each
edge of $H$ to both $\text{ec}(G-X)$ and $|X|$. Assume that an
edge $e$ of $H$ has been replaced with a string of $k$
diamonds. Since $k+1$ edges of the string belong to $X$, the
edge $e$ contributes $-1$ to the difference
$\text{ec}(G-X)-|X|$. There are $3n/2$ edges in $H$, so
$\text{ec}(G-X)-|X|=-3n/2$. Equation \eqref{eq:Nebesky} implies
that
$$2\ge 2\text{oc}(G-X)+\text{ec}(G-X)-|X|=2n-3n/2,$$
hence $n\le 4$. Recall that $G$ is two-face-embeddable, which
means that $|V(G)|\equiv 0\pmod4$. It follows that
$n=|V(H)|\equiv 0\pmod4$ as well, and therefore $n=4$. Thus $H=K_4$ or
$H=L_4$.

\begin{figure}[h]
\centering
\resizebox{0.95\textwidth}{!}{
\input{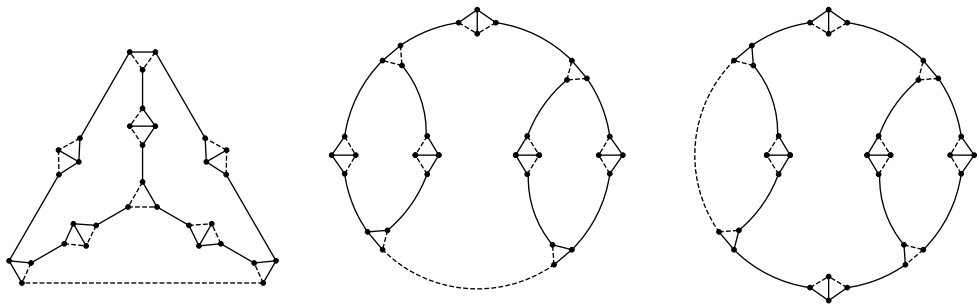}}
\caption{Heavy Xuong trees in three small claw-free graphs.}
    \label{fig:3-heavy}
\end{figure}

To finish the proof of the implication (i) $\Rightarrow$ (ii)
we need show that to obtain $G$ each edge of $H$ must be
replaced with a string of diamonds which has positive length.
Suppose this is not the case, and there are certain edges in
$H$ which are inherited to $G$ without any diamond insertion.
First, let us examine the case where there is only one such
edge $e_0$ in $H$. Up to isomorphism, there are two
possibilities to choose $e_0$ in $L_4$ and one in $K_4$. In all
three cases $G$ admits a heavy Xuong tree $T$ with the odd cotree
component containing $e_0$; see Figure~\ref{fig:3-heavy}. Hence,
by Theorem~\ref{thm:amply2-face},
$G$ is amply upper-embeddable, contrary to the
assumption. To proceed, observe that each diamond $D\subseteq
G$ supports an even component $Q$ of $G-E(T)$ such that
$Q\subseteq D$ and both edges of $\delta_G(D)$ are contained
in~$T$. If we contract any number of diamonds in $G$ and
suppress the resulting $2$-valent vertices, $T$ will transform
to a heavy Xuong tree of the resulting cubic graph, with the
original heavy cotree component containing $e_0$ being
preserved. It follows that $G$ is amply upper-embeddable unless
each edge of $H$ has been replaced with a string of diamonds
of positive length. Summing up, we have proved that if $G$ is a
simple $2$-connected claw-free cubic graph which is tightly
two-face-embeddable, then $G\in\mathcal{F}$.

\medskip\noindent
(ii) $\Rightarrow$ (i): Assume that $G\in\mathcal{F}$. We first
show that $G$ admits a light Xuong tree. As
Figure~\ref{fig:c2c_counter} indicates, this is the case for
both $F_1$ and $F_2$. Since $G$ arises from a graph
$F\in\{F_1,F_2\}$ by iterated diamond insertion, a light
spanning tree in $G$ is guaranteed by Lemma~\ref{lm:light}.

To prove that $G$ is tightly two-face embeddable we need to
show that every Xuong tree of $G$ is light. We proceed by
contradiction and suppose that $G$ contains a heavy Xuong tree
$T$. By Lemma~\ref{lm:amply}, we can assume that the cotree of
$T$ is acyclic. Let $C$ denote the corresponding cotree. We now
examine all possible ways of how $T$ and $C$ can intersect a
diamond or a vertex-triangle.

Firstly, we analyse the diamonds. It is clear that for each
diamond $D$ the intersection of $T$ with $\delta_G(D)$ consists
of either one or two edges. Accordingly, we distinguish several types of
diamonds three of which are of particular interest.

We say that a diamond $D$ is
\begin{itemize}
\item \emph{Type $1$}, if exactly one edge of $\delta_G(D)$
    belongs to $T$ and $C\cap D$ forms a path of length $2$
    connecting the two $2$-valent vertices of $D$;

\item \emph{Type $2$}, if both edges of $\delta_G(D)$
    belong to $T$ and $C\cap D$ forms a path connecting the
    two $2$-valent vertices of $D$; and

\item \emph{Type $3$}, if both edges of $\delta_G(D)$
    belong to $T$ and $C\cap D$ forms a path of length $3$
    connecting a $3$-valent vertex of $D$ to a $2$-valent
    vertex of $D$.
\end{itemize}

Next we discuss the vertex-triangles. It may be useful to
realise that $T$ contains at least one edge of each
vertex-triangle $W$ and at least one edge of each
$\delta_G(W)$. Two types of vertex-triangles are particularly
important.

We say that a vertex-triangle $W$ is
\begin{itemize}
\item \emph{Type $1$}, if $C\cap W$ is a path of length $2$
    and exactly one edge of $\delta_G(W)$ belongs to~$C$; and

\item \emph{Type $2$}, if $C\cap W$ is a path of length $2$
    and all three edges of $\delta_G(W)$ belong to $T$.
\end{itemize}
Note that in Type 1 the cotree edge of $\delta_G(W)$ must be
incident with the initial or terminal vertex of the path $C\cap
W$, otherwise $T$ fails to be a spanning subgraph.

The following claim explains the importance of the five types
of diamonds and vertex-triangles specified above.

\medskip\noindent
Claim. \emph{In $G$, there exists a heavy Xuong tree such that
each diamond is of Type~$1$, $2$, or~$3$, and each
vertex-triangle is of Type~$1$ or~$2$.}

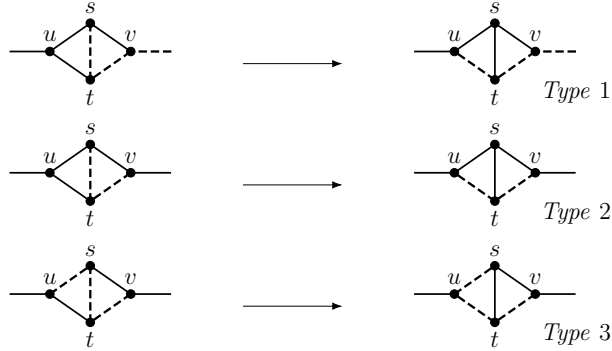
\begin{figure}[htbp]
\centering
\resizebox{0.6\textwidth}{!}{
\begin{tikzpicture}
[every node/.style={inner sep=0,outer sep=0}, line cap=rect, scale=0.7]  

  \node(a0) {};
  \path (a0)++ (2,0) node (a1){};
  \path (a0)++ (1,0) node (a2) {};
  \path (a0)++ (-1,0) node (a5) {};
  \path (a0)++ (-2,0) node (a6) {};  
  \path (a0)++ (90: 0.7) node(a3){};
  \path (a0)++ (270: 0.7) node(a4){};

  \filldraw (a2) circle (3pt)node[above, yshift=5] {$v$};
  \filldraw (a3) circle (3pt)node[above, yshift=5] {$s$};
  \filldraw (a4) circle (3pt)node[below, yshift=-5] {$t$};
  \filldraw (a5) circle (3pt)node[above, yshift=5] {$u$};

  \draw [dashed, line width=0.4mm](a1) -- (a2);  
  \draw [-, line width=0.3mm](a2) -- (a3);  
  \draw [dashed, line width=0.4mm](a2) -- (a4);   
  \draw [dashed, line width=0.4mm](a3) -- (a4);  
  \draw [-, line width=0.3mm](a4) -- (a5);  
  \draw [-, line width=0.3mm](a5) -- (a3); 
  \draw [-, line width=0.3mm](a5) -- (a6);  
%------------------------------------------------------------------------  

  \path(a0) ++(5.5,0) node (b0) {};
  \path(b0)++ (-1.74,-0.3) node (a8){};
  \path(b0)++ (0.75,-0.3) node (a9){};  
    \draw[-{Latex}] (a8) -- (a9)node[midway, above, yshift=2]{};
    
  \path(a0) ++(5.5,-3) node (b0) {};
  \path(b0)++ (-1.74,-0.3) node (a8){};
  \path(b0)++ (0.75,-0.3) node (a9){};  
    \draw[-{Latex}] (a8) -- (a9)node[midway, above, yshift=2]{};  
    
  \path(a0) ++(5.5,-6) node (b0) {};
  \path(b0)++ (-1.74,-0.3) node (a8){};
  \path(b0)++ (0.75,-0.3) node (a9){};  
    \draw[-{Latex}] (a8) -- (a9)node[midway, above, yshift=2]{};      
 
%------------------------------------------------------------------------  

  \path(a0) ++(10,0) node (c0) {};
  \path (c0)++ (2,0) node (a1){};
  \path (c0)++ (1,0) node (a2) {};
  \path (c0)++ (-1,0) node (a5) {};
  \path (c0)++ (-2,0) node (a6) {};  
  \path (c0)++ (90: 0.7) node(a3){};
  \path (c0)++ (270: 0.7) node(a4){};

  \path(c0)++ (2,-1) node (a8){\emph{Type} $1$};

  \filldraw (a2) circle (3pt)node[above, yshift=5] {$v$};
  \filldraw (a3) circle (3pt)node[above, yshift=5] {$s$};
  \filldraw (a4) circle (3pt)node[below, yshift=-5] {$t$};
  \filldraw (a5) circle (3pt)node[above, yshift=5] {$u$};

  \draw [dashed, line width=0.4mm](a1) -- (a2);  
  \draw [-, line width=0.3mm](a2) -- (a3);  
  \draw [dashed, line width=0.4mm](a2) -- (a4);   
  \draw [-, line width=0.3mm](a3) -- (a4);  
  \draw [dashed, line width=0.4mm](a4) -- (a5);  
  \draw [-, line width=0.3mm](a5) -- (a3); 
  \draw [-, line width=0.3mm](a5) -- (a6);

%------------------------------------------------------------------------  

  \path(a0) ++(0,-3) node (d0) {};
  \path (d0)++ (2,0) node (a1){};
  \path (d0)++ (1,0) node (a2) {};
  \path (d0)++ (-1,0) node (a5) {};
  \path (d0)++ (-2,0) node (a6) {};  
  \path (d0)++ (90: 0.7) node(a3){};
  \path (d0)++ (270: 0.7) node(a4){};

  \filldraw (a2) circle (3pt)node[above, yshift=5] {$v$};
  \filldraw (a3) circle (3pt)node[above, yshift=5] {$s$};
  \filldraw (a4) circle (3pt)node[below, yshift=-5] {$t$};
  \filldraw (a5) circle (3pt)node[above, yshift=5] {$u$};

  \draw [-, line width=0.3mm](a1) -- (a2);  
  \draw [-, line width=0.3mm](a2) -- (a3);  
  \draw [dashed, line width=0.4mm](a2) -- (a4);   
  \draw [dashed, line width=0.4mm](a3) -- (a4);  
  \draw [-, line width=0.3mm](a4) -- (a5);  
  \draw [-, line width=0.3mm](a5) -- (a3); 
  \draw [-, line width=0.3mm](a5) -- (a6);  

%------------------------------------------------------------------------  

%------------------------------------------------------------------------  
%  \path(a0) ++(0,-5) node (z0) {};
%------------------------------------------------------------------------  
%  \path(z0) ++(2,0) node (b0) {};
%------------------------------------------------------------------------  

  \path(c0) ++(0,-3) node (c0) {};
  \path (c0)++ (2,0) node (a1){};
  \path (c0)++ (1,0) node (a2) {};
  \path (c0)++ (-1,0) node (a5) {};
  \path (c0)++ (-2,0) node (a6) {};  
  \path (c0)++ (90: 0.7) node(a3){};
  \path (c0)++ (270: 0.7) node(a4){};

  \path(c0)++ (2,-1) node (a8){\emph{Type} $2$};

  \filldraw (a2) circle (3pt)node[above, yshift=5] {$v$};
  \filldraw (a3) circle (3pt)node[above, yshift=5] {$s$};
  \filldraw (a4) circle (3pt)node[below, yshift=-5] {$t$};
  \filldraw (a5) circle (3pt)node[above, yshift=5] {$u$};

  \draw [-, line width=0.3mm](a1) -- (a2);  
  \draw [-, line width=0.3mm](a2) -- (a3);  
  \draw [dashed, line width=0.4mm](a2) -- (a4);   
  \draw [-, line width=0.3mm](a3) -- (a4);  
  \draw [dashed, line width=0.4mm](a4) -- (a5);  
  \draw [-, line width=0.3mm](a5) -- (a3); 
  \draw [-, line width=0.3mm](a5) -- (a6);  

%------------------------------------------------------------------------  

  \path(d0) ++(0,-3) node (e0) {};
  \path (e0)++ (2,0) node (a1){};
  \path (e0)++ (1,0) node (a2) {};
  \path (e0)++ (-1,0) node (a5) {};
  \path (e0)++ (-2,0) node (a6) {};  
  \path (e0)++ (90: 0.7) node(a3){};
  \path (e0)++ (270: 0.7) node(a4){};

  \filldraw (a2) circle (3pt)node[above, yshift=5] {$v$};
  \filldraw (a3) circle (3pt)node[above, yshift=5] {$s$};
  \filldraw (a4) circle (3pt)node[below, yshift=-5] {$t$};
  \filldraw (a5) circle (3pt)node[above, yshift=5] {$u$};

  \draw [-, line width=0.3mm](a1) -- (a2);  
  \draw [-, line width=0.3mm](a2) -- (a3);  
  \draw [dashed, line width=0.4mm](a2) -- (a4);   
  \draw [dashed, line width=0.4mm](a3) -- (a4);  
  \draw [-, line width=0.3mm](a4) -- (a5);  
  \draw [dashed, line width=0.4mm](a5) -- (a3); 
  \draw [-, line width=0.3mm](a5) -- (a6);

%------------------------------------------------------------------------  

  \path(c0) ++(0,-3) node (f0) {};
  \path (f0)++ (2,0) node (a1){};
  \path (f0)++ (1,0) node (a2) {};
  \path (f0)++ (-1,0) node (a5) {};
  \path (f0)++ (-2,0) node (a6) {};  
  \path (f0)++ (90: 0.7) node(a3){};
  \path (f0)++ (270: 0.7) node(a4){};
  
  \path(f0)++ (2,-1) node (a8){\emph{Type} $3$};

  \filldraw (a2) circle (3pt)node[above, yshift=5] {$v$};
  \filldraw (a3) circle (3pt)node[above, yshift=5] {$s$};
  \filldraw (a4) circle (3pt)node[below, yshift=-5] {$t$};
  \filldraw (a5) circle (3pt)node[above, yshift=5] {$u$};

  \draw [-, line width=0.3mm](a1) -- (a2);  
  \draw [-, line width=0.3mm](a2) -- (a3);  
  \draw [dashed, line width=0.4mm](a2) -- (a4);   
  \draw [-, line width=0.3mm](a3) -- (a4);  
  \draw [dashed, line width=0.4mm](a4) -- (a5);  
  \draw [dashed, line width=0.4mm](a5) -- (a3); 
  \draw [-, line width=0.3mm](a5) -- (a6);  
%------------------------------------------------------------------------  

\end{tikzpicture}}
\caption{Transformation of diamonds to Type~1, 2, and 3.}
    \label{fig:diamonds}
\end{figure}

\medskip\noindent
Proof of Claim. We start with an arbitrary heavy Xuong tree
$T\subseteq G$ whose cotree $C$ is acyclic. Consider an
arbitrary diamond $D$ of~$G$; let $u$ and $v$ be the $2$-valent
vertices and let $s$ and $t$ be the $3$-valent vertices of $D$.
If the edge $st$ belongs to $T$, then $D$ it is easily seen to
be one of Type~1, 2, or~3. To see it, it is sufficient to
realise that $C$ is acyclic and its unique odd component is
heavy. If $st$ belongs to $C$ then, up to isomorphism, there
are three possibilities for the distribution of edges of $D$
into $T$ and $C$; they are illustrated in
Figure~\ref{fig:diamonds} on the left. In each of these three
cases one can find an edge $x$ of $D$ different from $st$ such
that the elementary switch $T'=T+st-x$ gives rise to a Xuong
tree whose cotree is again acyclic and has exactly one heavy
odd component. According to the notation introduced in
Figure~\ref{fig:diamonds}, it is sufficient to take $x=tv$ in
all three cases. Moreover, with respect to $T'$ the diamond $D$
turns to be one of Types~1, 2, or 3. By repeating this
procedure wherever necessary we eventually obtain a Xuong tree
with all diamonds of Type~1, 2, or 3.

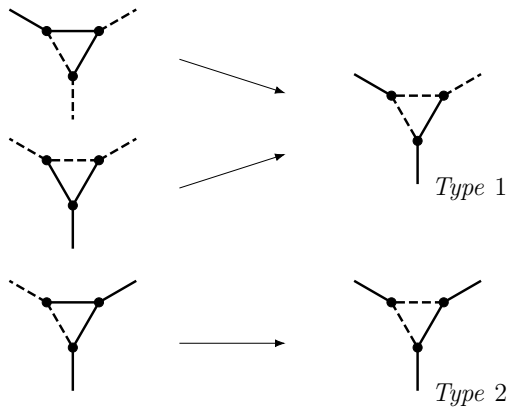
\begin{figure}[htbp]
\centering
\resizebox{0.5\textwidth}{!}{
 \begin{tikzpicture}
[every node/.style={inner sep=0,outer sep=0}, line cap=rect, scale=0.7]  

  \node(a0) {};
%------------------------------------------------------------------------  

  \path(a0) ++(0,0.3) node (f0) {};
  \path (f0)++ (30:0.7) node (a1){};
  \path (f0)++ (30:1.7) node (a2) {};
  \path (f0)++ (150: 0.7) node(a3){};
  \path (f0)++ (150: 1.7) node(a4){};
  \path (f0)++ (270:0.7) node (a5) {};
  \path (f0)++ (270:1.7) node (a6) {};  

  \filldraw (a1) circle (3pt);  
  \filldraw (a3) circle (3pt);
  \filldraw (a5) circle (3pt);

  \draw [dashed, line width=0.4mm](a1) -- (a2);
  \draw [-, line width=0.4mm](a1) -- (a3);
  \draw [-, line width=0.4mm](a3) -- (a4);
  \draw [dashed, line width=0.4mm](a3) -- (a5);    
  \draw [dashed, line width=0.4mm](a5) -- (a6);
  \draw [-, line width=0.4mm](a5) -- (a1);

  %------------------------------------------------------------------------  
 
   \path(f0) ++(0,-3) node (q0) {};
  \path (q0)++ (30:0.7) node (a1){};
  \path (q0)++ (30:1.7) node (a2) {};
  \path (q0)++ (150: 0.7) node(a3){};
  \path (q0)++ (150: 1.7) node(a4){};
  \path (q0)++ (270:0.7) node (a5) {};
  \path (q0)++ (270:1.7) node (a6) {};  

  \path(q0)++ (0.94,-2) node (a8){};

  \filldraw (a1) circle (3pt);  
  \filldraw (a3) circle (3pt);
  \filldraw (a5) circle (3pt);

  \draw [dashed, line width=0.4mm](a1) -- (a2);
  \draw [dashed, line width=0.4mm](a1) -- (a3); 
  \draw [dashed, line width=0.4mm](a3) -- (a4);  
  \draw [-, line width=0.4mm](a3) -- (a5);  
  \draw [-, line width=0.4mm](a5) -- (a6);
  \draw [-, line width=0.4mm](a5) -- (a1);  
 
%------------------------------------------------------------------------  
 
   \path(f0) ++(8,-1.5) node (f0) {};
  \path (f0)++ (30:0.7) node (a1){};
  \path (f0)++ (30:1.7) node (a2) {};
  \path (f0)++ (150: 0.7) node(a3){};
  \path (f0)++ (150: 1.7) node(a4){};
  \path (f0)++ (270:0.7) node (a5) {};
  \path (f0)++ (270:1.7) node (a6) {};  

  \path(f0)++ (1.2,-1.8) node (a8){\emph{Type} $1$};

  \filldraw (a1) circle (3pt);  
  \filldraw (a3) circle (3pt);
  \filldraw (a5) circle (3pt);

  \draw [dashed, line width=0.4mm](a1) -- (a2);  
  \draw [dashed, line width=0.4mm](a1) -- (a3); 
  \draw [-, line width=0.4mm](a3) -- (a4);
  \draw [dashed, line width=0.4mm](a3) -- (a5);  
  \draw [-, line width=0.4mm](a5) -- (a6);  
  \draw [-, line width=0.4mm](a5) -- (a1);

 %------------------------------------------------------------------------  

  \path(a0) ++(0,-6) node (f0) {};
  \path (f0)++ (30:0.7) node (a1){};
  \path (f0)++ (30:1.7) node (a2) {};
  \path (f0)++ (150: 0.7) node(a3){};
  \path (f0)++ (150: 1.7) node(a4){};
  \path (f0)++ (270:0.7) node (a5) {};
  \path (f0)++ (270:1.7) node (a6) {};  

  \filldraw (a1) circle (3pt);  
  \filldraw (a3) circle (3pt);
  \filldraw (a5) circle (3pt);

  \draw [-, line width=0.4mm](a1) -- (a2);  
  \draw [-, line width=0.4mm](a1) -- (a3);
  \draw [dashed, line width=0.4mm](a3) -- (a4);  
  \draw [dashed, line width=0.4mm](a3) -- (a5);  
  \draw [-, line width=0.4mm](a5) -- (a6);  
  \draw [-, line width=0.4mm](a5) -- (a1);

%------------------------------------------------------------------------  
 
   \path(f0) ++(8,0) node (f0) {};
  \path (f0)++ (30:0.7) node (a1){};
  \path (f0)++ (30:1.7) node (a2) {};
  \path (f0)++ (150: 0.7) node(a3){};
  \path (f0)++ (150: 1.7) node(a4){};
  \path (f0)++ (270:0.7) node (a5) {};
  \path (f0)++ (270:1.7) node (a6) {};

  \filldraw (a1) circle (3pt);  
  \filldraw (a3) circle (3pt);
  \filldraw (a5) circle (3pt);

   \path(f0)++ (1.2,-1.8) node (a8){\emph{Type} $2$};
  
  \draw [-, line width=0.4mm](a1) -- (a2);
  \draw [dashed, line width=0.4mm](a1) -- (a3); 
  \draw [-, line width=0.4mm](a3) -- (a4);
  \draw [dashed, line width=0.4mm](a3) -- (a5);  
  \draw [-, line width=0.4mm](a5) -- (a6);  
  \draw [-, line width=0.4mm](a5) -- (a1);  
  
 %------------------------------------------------------------------------  

  \path(a0) ++(4.2,0.3) node (b0) {};
  \path(b0)++ (-1.74,-0.3) node (a8){};
  \path(a8)++ (2.5,-0.8) node (a9){};  
    \draw[-{Latex}] (a8) -- (a9){};
    
  \path(a0) ++(4.2,-3.3) node (b0) {};
  \path(b0)++ (-1.74,0.3) node (a8){};
  \path(a8)++ (2.5,0.8) node (a9){};   
    \draw[-{Latex}] (a8) -- (a9){};
    
  \path(a0) ++(4.2,-6.3) node (b0) {};
  \path(b0)++ (-1.74,-0.3) node (a8){};
  \path(a8)++ (2.5,0) node (a9){};  
    \draw[-{Latex}] (a8) -- (a9){};

%------------------------------------------------------------------------  
\end{tikzpicture}}
\caption{Transformation of vertex-triangles to Type~1 and 2.}
    \label{fig:triangles}
\end{figure}

We now take care of vertex-triangles. By the previous part of
the proof we may assume that $T$ already has the property that
each diamond is of Type~1, 2, or 3. Consider an arbitrary
vertex-triangle $W$ of $G$. At least one edge of $W$ belongs to
$T$ and at least one edge of $\delta_G(W)$ belongs to $T$. If
exactly one edge of $W$ belongs to~$T$, then, up to
isomorphism, there are two possibilities for the distribution
of edges of $\delta_G(W)$ to $T$ and~$C$, which depend on the
size of $\delta_G(W) \cap E(T)$: if $|\delta_G(W)\cap E(T)|=2$,
then $W$ is Type~1, and if $|\delta_G(W)\cap E(T)|=3$, then $W$
is Type~2. The situation that $|\delta_G(W)\cap E(T)|=1$ does
not occur, as $T$ would not be connected. Next assume that
exactly two edges of $W$ belong to $T$. This leads to three
possible distributions, two with $|\delta_G(W)\cap E(T)|=1$ and
one with $|\delta_G(W)\cap E(T)|=2$, see
Figure~\ref{fig:triangles} on the left. The situation that
$|\delta_G(W)\cap E(T)|=3$ does not occur, because it would
create a light odd cotree component of $C$.

Observe that if $W$ is neither Type~1 nor Type~2, then at least
one edge of $\delta_G(W)$ lies in~$C$. If $z$ is such an edge,
then $z$ joins $W$ to a diamond of Type~1, implying that $C$
lies in a cotree component with at least three edges.
With this in mind, it is easy to perform a suitable elementary
switch of the form $T'=T+x-y$ in such a way that the resulting
tree $T'$ is again a heavy Xuong tree and, moreover, $W$ turns
to a vertex-triangle of Type~1 or Type~2. The transformations
are indicated in Figure~\ref{fig:triangles}. After performing
these modifications as many times as necessary we produce a
Xuong tree, still denoted by $T$, where each diamond is Type~1,
2, or~3, and each vertex-triangle is Type~1 or~2. The unique
odd component of the corresponding cotree remains heavy by all
theese modifications. This proves the claim.

\medskip

We are ready to derive a contradiction. For $i\in\{1,2,3\}$,
let $n_i$ denote the number of $i$-valent vertices of~$T$.  A
straightforward inductive argument implies that $n_1=n_3+2$.
Since the diamonds and the vertex-triangles partition the
vertex set of $G$, we can express $n_1$ and $n_3$ as sums of
values ranging through the set $\mathcal{D}$ of all diamonds
and the set $\mathcal{W}$ of vertex-triangles of $G$. Let $d_i$
denote the number of diamonds of Type~$i$, and let $w_i$ denote
the number of vertex-triangles of Type~$i$. Clearly,
$w_1+w_2=4$.

Note that every vertex-triangle $W$ of Type~1 is matched to a
unique diamond $D_W$ of Type~1 via a cotree edge $e_W$, and
together they form the subgraph $W\cup\{e_W\}\cup D_W$. There
exist $w_1$ such subgraphs in $G$, and since each of them
encloses a cotree component comprising five edges, we conclude
that $w_1\le 1$ and $w_2\ge 3$.

Given an induced subgraph $K\subseteq G$, set
$\lambda(K)=n_1(K)-n_3(K)$, where $n_i(K)$ denotes the number
of $i$-valent vertices of $T$ contained in $K$. In particular,
for a diamond $D$ we have $\lambda(D)=2-1=1$ if $D$ is Type~1;
$\lambda(D)=1-1=0$ if $D$ is Type~2; and $\lambda(D)=2-0=2$ if
$D$ is Type~3. Similarly, for a vertex-triangle $W$ we have
$\lambda(W)=2-0=2$ if $W$ is Type~1, and $\lambda(W)=1-0=1$ if
$W$ is Type~2. Putting this information together we obtain:
\begin{align*}
2=n_1-n_3&=\lambda(G)=
\sum_{D\in\mathcal{D}}\lambda(D)+\sum_{W\in\mathcal{W}}\lambda(W)\\
		 &=1d_1+0d_2+2d_3+2w_1+1w_2\ge w_2\ge 3,
\end{align*}
which is a contradiction. Our initial assumption that
$G\in\mathcal{F}$ admits a heavy Xuong tree was therefore
false; it follows that $G$ is tightly upper-embeddable. The
proof is complete.
\end{proof}

The collection of known tightly two-face-embeddable graphs can
be significantly enlarged by applying a specific form of a
well-known operation, known as the $2$-sum of cubic graphs.
Before defining this operation it is convenient to recall that
a cubic graph $G$ is tightly two-face-embeddable if and only if
every Xuong tree of $G$ is light, that is, the unique odd
component of the corresponding cotree is formed by a single
edge (Theorem~\ref{thm:amply2-face}). An edge $e$ of a tightly
two-face-embeddable cubic graph will be called \emph{odd} if
there exists a Xuong tree $T$ such that $e$ consitutes the
unique odd cotree component of $G-E(T)$.

Let $G_1$ and $G_2$ be two tightly two-face-embeddable graphs,
and let $e_i$ be an odd edge of $G_i$ for $i\in\{1,2\}$. We
construct a new graph $G$ by adding to $(G_1-e_1)\cup
(G_2-e_2)$ two independent edges $f_1$ and $f_2$, each joining
a $2$-valent vertex of $G-e_1$ to a $2$-valent vertex of
$G-e_2$, in such a way that $G$ becomes cubic. We say that $G$
is an \emph{odd $2$-sum} of $G_1$ and $G_2$ with respect to
$e_1$ and $e_2$. It is easy to see that $G$ is again
$2$-connected and that the two newly added edges form a
$2$-edge-cut of $G$. This cut will be referred to as the
\emph{principal} $2$-edge-cut of the $2$-sum.

\begin{theorem}\label{thm:tight2-sum}
An odd $2$-sum of two $2$-connected tightly two-face embeddable
cubic graphs is again $2$-connected and tightly two-face
embeddable.
\end{theorem}

\begin{proof}
Let $G_1$ and $G_2$ be tightly two-face-embeddable graphs with
odd edges $e_1$ and $e_2$, respectively, and let $G$ be an odd
$2$-sum of $G_1$ and $G_2$ with respect to $e_1$ and $e_2$. Let
$f_1$ and $f_2$ be the edges of the principal $2$-edge-cut of
$G$.

Take Xuong trees $T_1\subseteq G_1$ and $T_2\subseteq G_2$ for
which $e_1$ and $e_2$, respectively, form the corresponding odd
components. Observe that $(T_1\cup T_2)+f_1$ is a spanning tree
of $G$ and that $f_2$ constitutes the unique odd component of
the corresponding cotree. Thus $G$ is two-face-embeddable. To
prove that $G$ is tightly $2$-face-embeddable it remains to
show that for every Xuong tree $T$ of $G$ is light.

Suppose to the contrary that $G$ admits a Xuong tree $T$ such
that $G-E(T)$ has a heavy odd component, which we denote by
$H$. For $i\in\{1,2\}$ set $G_i'=G_i-e_i$ and $T_i=T\cap G_i'$.
Clearly, $T$ contains at least one of the edges $f_1$ and
$f_2$. Accordingly, we have two cases to consider.

\medskip\noindent
Case 1. \textit{The spanning tree $T$ contains exactly one edge
of the principal edge cut.} Without loss of generality we may
assume that $f_1$ is contained in $T$. It follows that
$T=T_1\cup\{f_1\}\cup T_2$ and both $T_1$ and $T_2$ are
spanning trees of $G_1$ and $G_2$, respectively. Now, if $H$
does not contain $f_2$, then $H$ is a heavy odd cotree component
with respect to either $T_1$ or $T_2$. However, this is
impossible because both $G_1$ and $G_2$ are tightly
upper-embeddable. Therefore $H$ contains $f_2$. For
$i\in\{1,2\}$ set $H_i=H\cap G_i'$. Since $H$ is odd and
contains $f_2$, both $H_1$ and $H_2$ have the same parity and
at least one of them is nonempty, say $H_1$. If $H_1$ is even,
then $H_1\cup\{e_1\}$ is a unique odd component of
$G_1-E(T_1)$, and is heavy, which is a contradiction. It
follows that both $H_1$ and $H_2$ are odd and consequently both
$T_1$ and $T_2$ have only even cotree components in $G_1$ and
$G_2$, respectively. This contradiction excludes Case~1.

\medskip\noindent
Case 2. \textit{The spanning tree $T$ contains both edges of
the principal edge cut.} In this case exactly one of $T_1$ and
$T_2$ is connected, say $T_1$. It follows that $T_1$ is a
spanning tree of $G_1$ and $T_2+e_2$ is a spanning tree of
$G_2$. Since $T$ contains both edges of the principal
$2$-edge-cut, each cotree component with respect to $T$ must be
contained either in $G_1'$ or in $G_2'$. If $H\subseteq G_2'$,
then $T_2+e_2$ would be a heavy Xuong tree of $G_2$, which is
impossible because $G_2$ is tightly two-face-embeddable.
Therefore $H\subseteq G_1'$. Now, $H$ cannot have a common
vertex with $e_1$, because $T_1$ would be a Xuong tree of $G_1$
with all cotree components even, one of them being $H\cup
\{e_1\}$. But if $H$ has no common vertex with $e_i$, then it
constitutes a heavy odd component of $G_1-E(T_1)$, which is
impossible because $G_1$ is tightly two-face-embeddable. Thus
Case~2 cannot happen as well, and the statement is proved.
\end{proof}

In certain cases, odd edges are not difficult to specify. For
instance, both the necklace $L_4$ and the graph of order $8$
obtained from the dipole $D_2$ by inserting a digon into each
edge are easily seen to be tightly two-face embeddable graphs.
It is also easy to see that each edge lying on a digon in any
of these two graphs is odd. Consulting
Figure~\ref{fig:c2c_counter} again we can conclude that in the
two basic graphs $F_1$ and $F_2$ of the family $\mathcal{F}$
each edge lying on a vertex-triangle is odd. Since
vertex-triangles are preserved by a diamond insertion and, by
Lemma~\ref{lm:light}, a tight spanning tree of the smaller
graph extends to a tight spanning tree of the larger one, we
conclude that every edge lying on a vertex-triangle of any
graph $F\in\mathcal{F}$ is also odd.

\section{Concluding remarks}

\begin{remark}
{\rm In Theorem~\ref{thm:amply2-face} we have characterised amply
two-face embeddable graphs as those which admit a Xuong tree
whose single odd cotree component has at least three edges.
Finding a similar
characterisation for amply one-face embeddable graphs -- or
even finding the ``right'' definition of ample one-face
embeddability that would be compatible with that for ample
two-face embeddability -- remains an open problem.}
\end{remark}

\begin{remark}
{\rm If we wish to prove that a given cubic graph $G$ is amply
two-face embeddable it is sufficient to find a heavy Xuong tree
in $G$. By contrast, the proof of
Theorem~\ref{thm:tight-family} suggests that to argue that $G$
is \emph{not} amply two-face-embeddable is not such an easy
task. The reason is that we do not have a tool similar to
Nebesk\'y's theorem, which can be efficiently used to prove
that a graph is not upper-embeddable. By
Equation~\eqref{eq:Nebesky}, a connected graph $G$ is
upper-embeddable if and only if
\begin{equation*}\label{eq:ue}
\text{ec}(G-X)+2\text{oc}(G-X)-2\le |X|
\end{equation*}
for each subset $X\subseteq E(G)$. In other words, to prove
that a connected cubic graph is not upper-embeddable it is
sufficient to identify a subset $Y\subseteq E(G)$ such that
$\text{ec}(G-Y)+2\text{oc}(G-Y)-2>|Y|$. In this context it is a
natural question to ask whether there exists a  function
$\alpha\colon E(G)\to\mathbb{Z}$ ``similar'' to the Nebesk\'y function
$\nu(X)= \text{ec}(G-X)+2\text{oc}(G-X)-2$
such that a connected cubic graph
is amply upper-embeddable (or at least amply
two-face-embeddable) if and only $\alpha(X)\le |X|$ for each
subset $X\subseteq E(G)$.}
\end{remark}

\begin{remark}
{\rm Theorem~\ref{thm:amply2-face} offers a natural question of
how amply and tightly upper-embeddable graphs are distributed
within the class of cubic graphs. In general, tightly
upper-embeddable cubic graphs are not easy to find. Moreover,
it does not seem easy to prove that a given cubic graph is
tightly upper-embeddable. We have shown that there exist
infinite families of tightly upper-embeddable cubic graphs with
connectivity $1$ and $2$. However, no examples of $3$-connected
tightly upper-embeddable graphs are known to us. This is why we
conjecture that a 3-connected cubic graph admits a coherent
decycling partition if and only if it is upper-embeddable
(Conjecture~\ref{conj}).

As far as cubic graphs with connectivity $2$ are concerned,
tight upper embeddability appears to be a rare event. This
indicates that among the upper-embeddable graphs those that are
tightly upper-embeddable should constitute a negligible part.
Therefore, if we take into account that almost all cubic graphs
are upper-embeddable and return to the equivalent language of
decycling partitions, it seems likely that
Problem~\ref{prob:as-amply} has a positive answer. In other
words, we make the following conjecture.}

\begin{conjecture}
{\rm Almost all cubic graphs contain an induced tree whose
removal leaves a subgraph with at most one edge. }
\end{conjecture}
\end{remark}

\section*{Acknowledgements}
\noindent{}The authors express their gratitude to the anonymous
referees for their careful reading and constructive
suggestions, and to J. Fiala and J. Karabáš for useful
comments.

\bibliographystyle{amsplain}
\bibliography{biblio_existence}

\end{document}